\documentclass[10pt,a4paper,leqno]{amsart}
\usepackage{amsmath}
\usepackage{amsthm}
\usepackage{amsfonts}
\usepackage{amssymb}
\usepackage{graphicx}
\usepackage{amsbsy}
\usepackage{mathrsfs}
\usepackage{esint}
\usepackage{bbm}
\usepackage{dsfont}
\usepackage{bm}
\usepackage{color}
\usepackage{mathabx}
\usepackage{mathtools}
\usepackage{eucal}
\usepackage{slashed}

\newcommand{\be}{\begin{equation}}
\newcommand{\ee}{\end{equation}}
\newcommand{\ba}{\begin{array}}
\newcommand{\ea}{\end{array}}
\newcommand{\bea}{\begin{eqnarray}}
\newcommand{\eea}{\end{eqnarray}}
\newcommand{\bee}{\begin{eqnarray*}}
\newcommand{\eee}{\end{eqnarray*}}

\newcommand{\MM}{\mathcal{M}}

\newcommand{\Om}{\Omega}

\newtheorem{theorem}{Theorem}
\newtheorem{lemma}{Lemma}[section]
\newtheorem{prop}{Proposition}[section]

\newtheorem*{remark*}{Remark}
\newtheorem*{remarks*}{Remarks}

\numberwithin{equation}{section}

\newcommand{\R}{\mathbb{R}}

\newcommand{\C}{\mathbb{C}}

\newcommand{\DD}{\Delta}
\newcommand{\pt}{\partial}
\renewcommand{\leq}{\leqslant}
\renewcommand{\geq}{\geqslant}

\newcommand{\weakto}{\rightharpoonup}

\newcommand{\eps}{\varepsilon}

\newcommand{\cO}{\mathcal O}

\renewcommand{\epsilon}{\varepsilon}

\renewcommand{\phi}{\varphi}

\catcode`@=11
\def\section{\@startsection{section}{1}%
  \z@{1.5\linespacing\@plus\linespacing}{.5\linespacing}%
  {\normalfont\bfseries\large\centering}}
\catcode`@=12

\begin{document}

\title{Blowup for Fractional NLS }

\begin{abstract}
We consider fractional NLS with focusing power-type nonlinearity
$$
i \pt_t u = (-\DD)^s u - |u|^{2 \sigma} u, \quad (t,x) \in \R \times \R^N,
$$
where $1/2< s < 1$ and $0 < \sigma < \infty$ for $s \geq N/2$ and $0 < \sigma \leq 2s/(N-2s)$ for $s < N/2$. We prove a general criterion for blowup of radial solutions in $\R^N$ with $N \geq 2$ for $L^2$-supercritical and $L^2$-critical powers $\sigma \geq 2s/N$. In addition, we study the case of fractional NLS posed on a bounded star-shaped domain $\Om \subset \R^N$ in any dimension $N \geq 1$ and subject to exterior Dirichlet conditions. In this setting, we prove a general blowup result without imposing any symmetry assumption on $u(t,x)$. 

For the blowup proof in $\R^N$, we derive a localized virial estimate for fractional NLS in $\R^N$, which uses Balakrishnan's formula for the fractional Laplacian $(-\DD)^s$ from semigroup theory. In the setting of bounded domains, we use a Pohozaev-type estimate for the fractional Laplacian to prove blowup.
\end{abstract}

\author[T. Boulenger]{Thomas Boulenger}
\address{University of Basel, Department of Mathematics and Computer Science, Spiegelgasse 1, CH-4051 Basel, Switzerland.}
\email{thomas.boulenger@unibas.ch}

\author[D. Himmelsbach]{Dominik Himmelsbach}
\address{University of Basel, Department of Mathematics and Computer Science, Spiegelgasse 1, CH-4051 Basel, Switzerland.}
\email{dominik.himmelsbach@unibas.ch}

\author[E. Lenzmann]{Enno Lenzmann}
\address{University of Basel, Department of Mathematics and Computer Science, Spiegelgasse 1, CH-4051 Basel, Switzerland.}
\email{enno.lenzmann@unibas.ch}

\maketitle

\section{Introduction and Main Results}

In this paper, we derive general criteria for blowup of solutions $u=u(t,x)$ for fractional NLS with focusing power-type nonlinearity given by 
\be \label{eq:fnls1}
i \pt_t u = (-\DD)^s u - |u|^{2 \sigma} u, \quad (t,x) \in \R \times \R^N.
\ee
Here the integer $N \geq 1$ denotes the space dimension, $(-\DD)^s$ stands for the fractional Laplacian with power $s \in (0,1)$, defined by its symbol $|\xi|^{2s}$ in Fourier space, and $\sigma > 0$ is a given exponent. The evolution problem \eqref{eq:fnls1} can be seen as a canonical model for a nonlocal dispersive PDE with focusing nonlinearity that can exhibit solitary waves, turbulence phenomena, and blowup of solutions (i.\,e.~singularity formation). We refer to \cite{CaMaMc-01, We-87, IoPu-14,KrLeRa-13,GuZh-11,ChHaHwOz-13, ChHwKwLe-13, KlSpMa-14} for a (non-exhaustive) list of studies of fractional NLS in mathematics, numerics, and physics.

Although problem \eqref{eq:fnls1} bears a strong resemblance to the well-studied classical NLS (corresponding to $s=1$), a general existence theorem for blowup solutions of problem \eqref{eq:fnls1} has remained a challenging open problem so far. To the best of the authors' knowledge, the cases that have been successfully addressed by now are: i) fractional NLS with nonlocal Hartree-type nonlinearities and radial data \cite{FrLe-07,ChHaHwOz-13}, and ii) a perturbative construction of minimal mass blowup solutions for the so-called focusing half-wave equation in $N=1$ dimension \cite{KrLeRa-13}. Despite these efforts, the existence of blowup solutions for the model case of fractional NLS with power-type nonlinearity has mainly remained elusive up to now, but  it has been strongly supported by numerical evidence \cite{KlSpMa-14}. In the present paper, we derive general blowup results for \eqref{eq:fnls1} in both the  $L^2$-supercritical and $L^2$-critical cases where $\sigma > 2s/N$ and $\sigma=2s/N$, respectively. In what follows, we shall discuss blowup for the fractional NLS \eqref{eq:fnls1} posed on all of $\R^N$ as well as on bounded domains. We treat these two cases separately as follows.

\subsection{Radial Blowup in $\R^N$}
We consider the initial-value problem
\be \tag{fNLS} \label{eq:fNLS}
\left \{ \begin{array}{ll} i \pt_t u = (-\DD)^s u - |u|^{2 \sigma} u, \\
u(0,x) = u_0(x) \in H^s(\R^N), \quad u : [0,T) \times \R^N \to \C. \end{array} \right .
\ee
Recall that we assume that $s \in (0,1)$, $\sigma > 0$, and $N \geq 1$ denotes the space dimension. In what follows, we shall assume that we are given a sufficiently regular solution $u(t)$. More precisely, that $u \in C([0,T); H^{2s}(\R^N))$ for reasons explained below. Let us mention that the local well-posedness theory for the range of $s \in (0,1)$, $N \geq 1$, and exponents $\sigma > 0$ considered below is not completely settled yet; see, e.\,g., \cite{HoSi-15,GuZh-11} for local well-posedness results for non-radial and radial data, respectively.

The evolution problem \eqref{eq:fNLS} shares many obvious similarities with the classical nonlinear Schr\"odinger equation. In particular, we have the formal conservation laws for the {\em energy} and the {\em $L^2$-mass} given by
\be
E[u] = \frac{1}{2} \int_{\R^N}  |(-\DD)^{s/2} u|^2 \, dx -  \frac{1}{2 \sigma+2} \int_{\R^N} |u|^{2 \sigma+2} \, dx, \quad M[u] = \int_{\R^N} |u|^2 \, dx.
\ee
In view of these conserved quantities and scaling properties of \eqref{eq:fNLS}, it is convenient to introduce the {\em scaling index} defined as
\be
s_c = \frac{N}{2} - \frac{s}{\sigma}.
\ee
Reflecting the scaling properties of \eqref{eq:fNLS} and the conservation of $M[u]$, we refer to the cases $s_c < 0$, $s_c=0$, and $s_c > 0$ as $L^2$-subcritical, $L^2$-critical, and $L^2$-supercritical, respectively. Furthermore, in analogy to classical NLS, we can use (formally at least) the conserved quantities $E[u]$ and $M[u]$ together with a sharp Gagliardo-Nirenberg inequality \eqref{ineq:GN} to conclude that $H^s$-valued solutions $u(t)$ are always a-priori bounded in the $L^2$-subcritical case $s_c < 0$. Thus we can expect that $H^s$-valued solutions $u(t)$ may blowup in finite (or infinite) time only if $s_c \geq 0$ holds. Moreover, guided by a further analogy to classical NLS, we expect that sufficient criteria for blowup of $u(t)$ can be found in terms of quantities of {\em ground states} $Q \in H^s(\R^N)$ that optimize the Gagliardo-Nirenberg inequality \eqref{ineq:GN} and satisfy
\be \label{eq:Q1}
(-\DD)^s  Q + Q - Q^{2 \sigma +1} = 0 \quad \mbox{in $\R^N$}
\ee 
in the energy-subcritical case $s_c < s$. In the energy-critical case $s=s_c$ (which needs $N > 2s$), the relevant object $Q \in \dot{H}^s(\R^N)$ is the ground state, which is the optimizer for the Sobolev inequality \eqref{ineq:Sob} normalized such that it holds
\be \label{eq:Q2}
(-\DD)^s Q - Q^{\frac{N+2s}{N-2s}} = 0 \quad \mbox{in $\R^N$}.
\ee
Uniqueness (modulo symmetries) of ground states $Q \in H^s(\R^N)$ for equation \eqref{eq:Q1} and all $s < s_c$ and any $N \geq 1$ was recently shown in \cite{FrLe-13, FrLeSi-15}, whereas uniqueness (modulo symmetries) of ground states $Q \in \dot{H}^s(\R^N)$ for \eqref{eq:Q2} is a classical fact due to Lieb \cite{Li-83}.

Our first main result indeed establishes a sufficient criterion for blowup of radial solutions for $0 \leq s_c \leq s$ in terms of the corresponding ground state $Q$. 
 
\begin{theorem} \label{thm:L2super}
Let  $N \geq 2$, $s \in (\frac 1 2, 1)$, $0 \leq s_c \leq s$ with $\sigma < 2s$. Assume that $u \in C \big ([0,T); H^{2s}(\R^N)\big )$ is a radial solution of \eqref{eq:fNLS}. Furthermore, we suppose that either
$$
E[u_0] < 0,
$$
or, if $E[u_0] \geq 0$, we assume that
$$
\left \{ \begin{array}{ccc} 
E[u_0]^{s_c} M[u_0]^{s-s_c} & <  &E[Q]^{s_c} M[Q]^{s-s_c} , \\[1ex]
\| (-\DD)^{s/2} u_0 \|_{L^2}^{s_c} \| u_0 \|_{L^2}^{s-s_c} & > & \| (-\DD)^{s/2} Q \|_{L^2}^{s_c} \| Q \|_{L^2}^{s-s_c} .
\end{array}
\right .
$$
Then the following conclusions hold.
\begin{itemize}
\item[(i)] {\bf $L^2$-Supercritical Case:} If $0 < s_c \leq s$, then $u(t)$ blows up in finite time in the sense that $T < +\infty$ must hold.
\item[(ii)] {\bf $L^2$-Critical Case:} If $s_c=0$, then $u(t)$ either blows up in finite time in the sense that $T < +\infty$ must hold, or $u(t)$ blows up infinite time such that
$$
\| (-\DD)^{s/2} u(t) \|_{L^2} \geq C t^{s} \quad \mbox{for all $t \geq t_*$},
$$ 
with some constants $C > 0$ and $t_* > 0$ that depend only on $u_0, s, N$.
\end{itemize}
\end{theorem}

\begin{remarks*}
{\em 
1) The condition $\sigma < 2s$ is technical; see the proof of Theorem \ref{thm:L2super} for details.

2) In the energy-critical case $s=s_c$, it may happen that $Q \not \in L^2(\R^N)$ and thus $M[Q] = +\infty$; see Section \ref{sec:cutoff} below. In this case, we use the convention $(+\infty)^0 = 1$. Hence the second blowup condition above becomes $E[u_0] < E[Q]$ and $\| (-\DD)^{s/2} u_0 \|_{L^2} > \| (-\DD)^{s/2} Q \|_{L^2}$ when $s=s_c$.

3) In the $L^2$-critical case $s_c=0$, the second blowup condition stated above is void, since we then get $M[u_0] < M[Q]$ and $M[u_0] > M[Q]$, which is impossible. Thus for $s_c=0$ the only admissible condition is $E[u_0] < 0$.  

4) We prefer to work with strong $H^{2s}$-valued solutions $u(t)$ of \eqref{eq:fNLS}, since we do not a have full-fledged local well-posedness theory for \eqref{eq:fNLS} at our disposal, which would allow us to cover the case of $H^{s}$-valued solutions $u(t)$ by approximation arguments in the estimates derived below. 

5) Note that we exclude the half-wave case $s=1/2$, which is due to the lack of control for the pointwise decay of a radial function $u \in H^{1/2}(\R^N)$ with $N \geq 2$. See also the remark following Theorem \ref{thm:dom} below about the half-wave case  on bounded domains.

6) The condition $s_c \leq s$ will be needed at a certain step in the proof below. However, the rest of the arguments carry over to negative energy solutions in the energy-supercritical range $s_c > s$ in a verbatim way. 

7) We refer to \cite{HoRo-07} for the idea of using the scale-invariant quantity $E[u_0]^{s_c} M[u_0]^{s-s_c}$ for blowup for classical NLS. See also \cite{Gu-14}.

8) We refer to the recent work \cite{ChHwSh-15} for a Kenig-Merle-type analysis of the energy-critical case $s_c=s$, where also a conditional result on the existence of type II blowup is given. 
}
\end{remarks*}

\subsection*{Comments on the Proof of Theorem \ref{thm:L2super}}
By integrating \eqref{eq:fNLS} against $i( x \cdot \nabla + \nabla \cdot x) \overline{u}(t)$ on $\R^N$, we make the observation that any sufficiently regular and spatially localized solution $u=u(t,x)$ of \eqref{eq:fNLS} satisfies the virial identity
\be \label{eq:virial}
\frac{d}{dt} \left ( 2 \, \mathrm{Im} \, \int_{\R^N} \overline{u}(t) x \cdot \nabla u(t) \, dx \right ) = 4 \sigma N E[u_0] - 2 (\sigma N - 2s)  \| (-\DD)^{s/2} u(t) \|_{L^2}^2 .
\ee
This law can be regarded as a differential expression reflecting the scaling properties of \eqref{eq:fNLS}, similar to the celebrated {\em Pohozaev identities} that occur in nonlinear elliptic PDE used to rule out nontrivial solutions in supercritical cases.\footnote{Actually, we will exploit this connection for (fNLS) on a bounded domain; see Theorem \ref{thm:dom} below.} However, the virial identity {\em per se} does not offer enough information to deduce singularity formation for solutions with negative energy $E[u_0] < 0$ in the $L^2$-critical and $L^2$-supercritical cases when $\sigma \geq 2s/N$. So far two methods have successfully been used to prove blowup results.

\subsubsection*{Coupling to a Variance Law} For classical NLS (i.\,e.~when $s=1$) we have the {\em Variance-Virial Law}, which can be expressed as
\be \label{eq:variance}
\frac{1}{2} \frac{d}{dt} \left ( \int_{\R^N} |x|^2 |u(t)|^2 \,dx \right ) = 2 \, \mathrm{Im} \left ( \int_{\R^N} \overline{u}(t) x \cdot \nabla u(t) \, dx \right ),
\ee
provided that $\int_{\R^N} |x|^2 |u_0|^2 \, dx < +\infty$ holds. By combining \eqref{eq:virial} and \eqref{eq:variance}, we obtain Glassey's celebrated blowup result for classical NLS with negative energy $E[u_0] < 0$ and finite variance (see, e.\,g.~\cite{SuSu-99} for a textbook discussion). However, this argument breaks down for $s \neq 1$, since identity \eqref{eq:variance} fails in this case, as one readily checks by dimensional analysis. Rather, it turns out that the suitable generalization of the variance for fractional NLS is given by the nonnegative quantity
\be
\mathcal{V}^{(s)}[u(t)] = \int_{\R^N} \overline{u}(t) x \cdot  (-\DD)^{1-s} x u(t) \,dx =  \| x (-\DD)^{\frac{1-s}{2}} u(t)  \|_{L^2}^2.
\ee
Given any sufficiently regular and spatially localized solution $u(t)$ of the {\em free} fractional Schr\"odinger equation $i \pt_t u = (-\DD)^s u$, a calculation yields the equation
\be \label{eq:var2}
\frac{1}{2s} \frac{d}{dt} \mathcal{V}^{(s)}[u(t)] = 2 \, \mathrm{Im} \, \left ( \int_{\R^N} \overline{u}(t) x \cdot \nabla u(t) \, dx \right ).
\ee
However, the use of $\mathcal{V}^{(s)}[u(t)]$ brings in serious complications in the nonlinear setting when $s \neq 1$. First of all, the identity \eqref{eq:var2} breaks down and the correct equation acquires highly nontrivial error terms due to the nonlinearity. In particular, for $s \in (0,1)$, these error terms seem very hard to control for local nonlinearities with $f(u) = -|u|^{2 \sigma} u$ being the model case, even in the class of radial solutions. So far, the only known cases where the application of $\mathcal{V}^{(s)}[u(t)]$ has turned out to be successful to prove blowup results for fractional NLS deal with radial solutions and focusing Hartree-type nonlinearities, e.\,g., $f(u) = -(|x|^{-\gamma} \ast |u|^2 ) u$ with $\gamma \geq 1$; see \cite{FrLe-07, FrLe-07b,ChHaHwOz-13, ChHwKwLe-15}. See also \cite{BoLe-15}, where a localized version of $\mathcal{V}^{(s)}[u(t)]$ is used to show blowup for biharmonic NLS (corresponding to $s=2$) with local nonlinearities by using some smoothing properties of $(-\DD)^{\frac{1-s}{2}}$ when $s > 1$.
 
\subsubsection*{Localized Virial Law} Another method for proving blowup results, and which by-passes the use of a variance-type quantity, is to replace the unbounded function $x$ by a suitable cutoff function $\phi_R$ such that $\nabla \phi_R(x) \equiv x$ for $|x| \leq R$ and $\nabla \phi_R(x) \equiv \mbox{const}$ for $|x| \gg R$. To the best of our knowledge, the approach goes to  Ogawa and Tsutsumi \cite{OgTs-91}, where blowup for radial solutions (with infinite variance) of $L^2$-supercritical focusing classical NLS is proved. (See also \cite{Me-96} for a use of localized virial identities to show blowup for the Zakharov system.)

In fact, it is the strategy of localized virial identities that we implement for fractional NLS to prove Theorem \ref{thm:L2super}. However, when one tries to directly apply the arguments in \cite{OgTs-91} to study the time evolution of $\MM_{\phi_R}[u(t)]$ for  fractional NLS, one encounters severe difficulties due to the nonlocal nature of $(-\DD)^s$. In particular, the nonnegativity of certain error terms due to the localization, which are pivotal in the arguments of \cite{OgTs-91}, seem to be elusive. To overcome this difficulty, we employ the representation formula
\be \label{eq:Balakintro}
(-\DD)^s = \frac{\sin \pi s}{\pi} \int_0^\infty m^{s-1} \frac{-\DD}{-\DD+m} \, dm,
\ee
valid for all $s \in (0,1)$, which is also known as {\em Balakrishnan's formula} used in semigroup theory (see, e\,g., \cite{Ba-60,Pa-83}). In fact, by means of \eqref{eq:Balakintro}, we are able to derive to the differential estimate 
\be \label{eq:Mineq}
\begin{aligned}
 \frac{d}{dt} \MM_{R}[u(t)]  & \leq 4  \sigma N E[u_0] - 2 \delta \| (-\DD)^{s/2} u(t) \|_{L^2}^2 \\
 & \quad + o_R(1) \cdot \left (1 +  \|(-\DD)^{s/2} u(t) \|_{L^2}^{(\sigma/s)+} \right ) ,
 \end{aligned}
\ee
for any sufficiently regular and radial solution $u(t,x)$ of \eqref{eq:fNLS} in dimensions $N \geq 2$ and $s \in (1/2,1)$. Here $\delta = \sigma N -2s > 0$ is a positive constant when $s_c>0$, and the error term $o_R(1)$ tends to $0$ as $R \to \infty$ uniformly in $t$. With the help of the key estimate \eqref{eq:Mineq}, we can then apply a standard comparison ODE argument to show that $u(t)$ cannot exist for all times $t \geq 0$ under the assumptions of Theorem \ref{thm:L2super}. For the $L^2$-critical case $s_c=0$ and hence $\delta=0$, the differential estimate \eqref{eq:Mineq} needs to be refined and leads only to the weaker conclusion as stated in Theorem \ref{thm:L2super} (ii).

Finally, let us also mention that the strategy of the proof of Theorem \ref{thm:L2super} can be carried over to radial solutions $u(t)$ of fractional NLS of the form
$$
i \pt_t u = L u + f(u),
$$
where $f(u)$ is a local or Hartree-type nonlinearity that satisfies appropriate conditions (e.\,g.~focusing and $L^2$-supercritical or $L^2$-critical). Moreover, the dispersive symbol $L$ can be of the form $L= (-\DD)^{s_1} + (-\DD)^{s_2}$ with $s_1, s_2 \in (1/2,1)$ or $L=(-\DD+1)^{s}$ with $s \in (1/2,1)$.

\subsection{Blowup on Bounded Domains} As our second main result, we establish a general  blowup result for fractional NLS that are posed on a bounded domain $\Om \subset \R^N$ with $N \geq 1$. [In particular, the following discussion can be applied to the one-dimensional case $N=1$ when $\Omega = (a,b) \subset \R$ is a bounded open interval.] Here the fractional Laplacian $(-\DD)^s$ on $\Omega$ will be supplemented with the so-called exterior Dirichlet condition on $\R^N \setminus \Om$.  In fact, this is a natural choice in view of applications in physics and probability. (Another non-equivalent definition of the fractional Laplacian would be $A^s= (-\DD |_{\mathrm{Dir}})^s$ by using the spectral calculus for the Dirichlet Laplacian $(-\DD) |_{\mathrm{Dir}}$ on $\Om$. We hope to discuss the setting with $A^s$ in future work.) 
 
Let $\Om \subset \R^N$ with $N \geq 1$ be a smooth bounded domain. We consider the fractional NLS with focusing power-type nonlinearity posed on $\Om \subset \R^N$ given by the initial-value problem
\be \tag{fNLS$_\Omega$}  \label{eq:fnlsO}
\left \{ \begin{array}{ll} i \pt_t u = (-\DD)^s  u - |u|^{2 \sigma} u, & \quad \mbox{for $x \in \Om$ and $t \in [0,T)$}, \\
u(0,x) = u_0(x) , & \quad \mbox{for $x \in \Om$ and $t =0$}, \\
u(t,x) = 0 , & \quad \mbox{for $x \in \R^N \setminus \Om$ and $t \geq 0$}.
\end{array} \right .
\ee
To give a proper definition of the fractional Laplacian $(-\DD)^s$ on $\Om$ appearing above, we collect some functional analytic facts from the literature concerning the fractional Laplacian $(-\DD)^s$ on $\Om$ with exterior Dirichlet conditions. For any $s \geq 0$, we introduce the space of $H^s$-functions in $\R^N$ that vanish outside the set $\Om$, which we denote by
\be \label{def:Hs0}
H_0^s(\Om) := \left \{ u \in H^s(\R^N) : \mbox{$u(x) = 0$ for a.\,e.~$x \in \R^N \setminus \Om$} \right \} .
\ee
It can be shown that 
\be
\mathcal{Q}(u,v) := \int_{\Om} \overline{u} (-\DD)^s v \, dx
\ee
is a closed nonnegative symmetric quadratic form with form domain $H_0^s(\Om) \subset L^2(\Om)$. By standard operator theory, there is a unique self-adjoint operator $L=L^* \geq 0$ such that $\mathcal{Q}(u,v) = \langle u, L v \rangle$ for all $u, v \in H^s_0(\Om)$, where $\langle f, g \rangle = \int_{\Om} \overline{f} g \, dx$ denotes the inner product on $L^2(\Om)$. For notational simplicity, we shall often write $L=(-\DD)^s$
 in the following (and we thus skip the dependence of $L$ on the domain $\Om$). Furthermore, let us denote the operator domain of $(-\DD)^s$ by
\be
D((-\DD)^s) = \{ u \in H^s_0(\Om) : (-\DD)^s u \in L^2(\Om) \},
\ee   
endowed with the operator norm $\| u \|_{D((-\DD)^s)}^2 = \| u \|_{L^2(\Om)}^2 + \| (-\DD)^s u \|_{L^2(\Om)}^2$.  Furthermore, by standard theory, it follows that the spectrum of $L$ is discrete and given by a nondecreasing sequence of eigenvalues $0 < \lambda_1 < \lambda_2 \leq \lambda_3 \leq \ldots$ such that $\lambda_k \to +\infty$ as $k \to \infty$. In particular, the resolvent $L^{-1}$ is a compact operator on $L^2(\Om)$. Since $L$ is self-adjoint, it follows from Stone's theorem that $L$ generates a unitary group of isometries $\{ e^{-it L} \}_{t \in \R}$ on any of the Hilbert spaces $X \in \{ L^2(\Om), H_0^s(\Om), D(L) \}$. Hence we say that $u \in C([0,T); X)$ is a solution of \eqref{eq:fnlsO} if $u(t)$ solves the corresponding integral equation
\be
u(t) = e^{-it L} u_0 + i \int_0^t e^{-i(t-s)L} ( |u(s)|^{2 \sigma} u(s)) \, ds \quad \mbox{for $t \in [0,T)$}.
\ee
As in the case $\R^N$, we shall not study the local well-posedness theory for \eqref{eq:fnlsO} in the spaces $D(L^{1/2})=H^s_0(\Om)$ or $D(L)=D((-\DD)^s)$. 
 
Whereas the characterization of the form domain $D(L^{1/2})=H^s_0(\Om)$ is simple, the study of the operator domain $D(L)=D((-\DD)^s)$ turns out to be rather intricate. In her recent work \cite{Gr-15}, Grubb has proven (by partly building upon work of H\"ormander \cite{Ho-65}) that the operator domain is given by
$$
D((-\DD)^s) = H^{s(2s)}_2(\overline{\Om}),
$$
where  $H^{\mu(\nu)}_2(\overline{\Om})$ denotes the so-called $\mu$-transmission Sobolev space introduced by H\"ormander, indexed by $\mu \in \C$ and $\nu \in \R$ satisfying $\nu > \mathrm{Re} \, \mu -1/2$. But we will not be concerned with the fine properties of the spaces $H^{\mu(\nu)}_2$ as provided in \cite{Gr-15}; in particular, we only need the embedding $H^{s(2s)}_2(\overline{\Om}) \subset H^1_0(\Om)$ when $s > 1/2$.

Let us now assume that $u \in C([0,T); D((-\DD)^s))$ is a solution of \eqref{eq:fnlsO}. A well-defined calculation using the regularity of $u(t)$ then yields conservation of {\em energy}\footnote{Formally, the conservation of energy also holds for solutions $u \in C([0,T);  H_0^s(\Om))$. However, since we do not study the local well-posedness and approximation theory here, we rather prefer to work with operator domain-valued solutions $u \in C([0,T); D((-\DD)^s))$.}
\be
E_{\Om}[u(t)] = \frac{1}{2} \int_{\Om} \overline{u}(t) (-\DD)^s u(t) \, dx - \frac{1}{2 \sigma+2} \int_{\Om} |u(t)|^{2 \sigma+2} \, dx ,
\ee
and {\em $L^2$-mass} given by
\be
M_{\Om}[u(t)] = \int_{\Om} |u(t)|^2 \, dx .
\ee 

We can now  state the following blowup result concerning problem \eqref{eq:fnlsO} for star-shaped (in particular, convex) domains $\Om$.

\begin{theorem} \label{thm:dom}
Let $N \geq 1$, $s \in (1/2,1)$, and $0 < s_c \leq s$. Assume that $\Om \subset \R^N$ is a bounded and star-shaped domain with smooth boundary $\pt \Om$. Suppose that $u \in C \big ([0,T); D((-\DD)^s) \big )$ is a solution of \eqref{eq:fnlsO} with negative energy
$$E_\Om[u_0] < 0.$$ 
Then $u(t)$ blows up in finite time in the sense that $T < +\infty$ must hold.
\end{theorem}

\begin{remarks*}
{\em 1) In contrast to Theorem \ref{thm:L2super}, we do not impose a symmetry condition on $u(t)$. In addition, the one-dimensional case $N=1$ when $\Om = (a,b)$ is a bounded open interval is covered.

2) The proof of Theorem \ref{thm:dom} extends  {\em formally (at least)} to the half-wave case $s=1/2$ and leads to an infinite-time blowup result with exponential growth. However, a delicate domain/regularity issue of $D((-\DD)^s)$ for $s=1/2$ prevents us from  doing so. For more details, see the remark following the proof of Theorem \ref{thm:dom} below. }
\end{remarks*}

\subsection*{Comments on the Proof of Theorem 2}
The proof of Theorem \ref{thm:dom}  uses the time evolution of the full virial 
$$
\MM_{\Om}[u(t)] = 2 \, \mathrm{Im} \int_{\Om} \overline{u}(t) (x \cdot \nabla u(t)) \, dx.
$$
Since $x$ is a bounded function on $\Om$, there is no need to introduce a cutoff function and, moreover, we do not have to use spatial decay estimates for $u(t)$ (and hence impose radiality), since $\Om$ is bounded. However, the study of the time derivative of $\MM_{\Om}[u(t)]$ will involve a boundary term, whose sign will turn out to be  favorable if $\Om$ is star-shaped. This is a similar observation used in \cite{Ka-87} where blowup for classical NLS posed on domains is proved. A delicate point in the argument is to have the right substitute for an integration by parts formula for the nonlocal operator $(-\DD)^s$ on $\Om$. To handle this, we make use of a recent idea developed by X.~Ros-Oton and J.~Serra \cite{RoSe-14}, where a Pohozaev identity for the fractional Laplacian on bounded domains was derived.

\subsection*{Notation and Conventions} We write $X \lesssim Y$ to denote that $X \leq C Y$ with some constant $C >0$ that only depends on the fixed quantities $u_0$, $N$, $s$, $\sigma$, and some fixed cutoff function. Moreover, we employ the notation $X = \cO(Y)$ by which we mean that $|X| \lesssim Y$ holds. We use the standard convention by summing over repeated indices, e.\,g.,  ~$x_i y_i \equiv \sum_{i=1}^N x_i y_i$.

\subsection*{Acknowledgments} The authors gratefully acknowledge financial support by the Swiss National Science Foundation (SNF) through Grant No.~200021--149233. E.\,L.~also thanks G.~Grubb for a helpful correspondence about her results in \cite{Gr-15}.

\section{Localized Virial Estimate for Fractional NLS}

In this section, we derive localized virial estimates for radial solutions of fractional NLS. First, we derive a general virial formula for solutions $u(t,x)$ that are not necessarily radial. Then, we sharpen the estimates in the class of radial solutions.

\subsection{A General Virial Identity}
Let $N \geq 1$, $s \in [1/2,1)$, and $\sigma > 0$. Throughout the rest of this section, we assume that 
$$
u \in C \big ([0,T); H^{2s}(\R^N) \cap L^{2 \sigma +2}(\R^N) \big)
$$
is a solution of \eqref{eq:fNLS}. Note that, at this point, we do not impose any symmetry assumption on the solution $u(t,x)$. Note also that for $u(t) \in H^{2s}(\R^N)$, conservation of energy $E[u]$ and $M[u]$ follows directly by integrating the equation against $\pt_t {\overline{u}}(t)$ and $\overline{u}(t)$, respectively. There is no need for an approximation argument in order to have well-defined pairings.

Of course, if the exponent $\sigma$ is not $H^{2s}$-supercritical (in particular if $s_c \leq s$), the condition $u \in C([0,T); L^{2 \sigma+2}(\R^N))$ is superfluous by Sobolev embeddings. Furthermore, we remark the following localized virial identities could be extended to $u \in C([0,T); H^{s}(\R^N))$, provided we have a decent local well-posedness theory in $H^s(\R^N)$. However, as pointed out in the introduction, we prefer to work with strong $H^{2s}$-valued solutions for \eqref{eq:fNLS} in order to guarantee that the following calculations are well-defined a-priori.

Let us assume that $\phi : \R^N \to \R$ is a real-valued function with $\nabla \phi \in W^{3,\infty}(\R^N)$. We define the {\bf localized virial} of $u=u(t,x)$ to be the quantity given by
\be
\MM_{\phi}[u(t)] := 2 \, \mathrm{Im}  \int_{\R^N} \overline{u}(t) \nabla \phi \cdot \nabla u(t) \, dx = 2 \, \mathrm{Im} \int_{\R^N} \overline{u}(t) \pt_k \phi \pt_k u(t) \, dx .
\ee
Recall that we use the convention by summing over repeated indices from $1$ to $N$. By applying Lemma \ref{lem:MR_estimate}, we obtain the bound 
$$
|\MM_{\phi}[u(t)]| \lesssim C(\| \nabla \phi \|_{L^\infty}, \| \Delta \phi \|_{L^{\infty}}) \| u(t) \|_{H^{1/2}}^2 .
$$
Hence the quantity $\MM_{\phi}[u(t)]$ is well-defined, since $u(t) \in H^{s}(\R^N)$ with some $s \geq 1/2$ by assumption. 

To study the time evolution of $\MM_{\phi}[u(t)]$, we shall need the following auxiliary function $u_m=u_m(t,x)$ that is defined as
\be \label{def:um}
u_m (t) := c_s \frac{1}{-\DD + m} u(t) = c_s \mathcal{F}^{-1} \left ( \frac{\widehat{u}(t,\xi)}{|\xi|^2+m}  \right ) \quad \mbox{with $m >0$} ,
\ee
where the constant 
\be \label{def:cs}
c_s := \sqrt{ \frac{\sin \pi s}{\pi }}
\ee 
turns out to be a convenient normalization factor. By the smoothing properties of $(-\DD +m)^{-1}$, we clearly have that $u_m(t) \in H^{\alpha + 2}(\R^N)$ holds for any $t \in [0,T)$ whenever $u(t) \in H^\alpha(\R^N)$.

\begin{lemma} \label{lem:virial_gen}
For any $t \in [0,T)$, we have the identity
\begin{align*}
\frac{d}{dt} \MM_{\phi}[u(t)] & =    \int_0^\infty m^s  \int_{\R^N} \left \{  4   \, \overline{\pt_k u_m} (\pt^2_{kl} \phi) \pt_l u_m  -   (\DD^2 \phi)  |u_m|^2 \right \} \, dx \, d m \\
& \quad -  \frac{2 \sigma}{\sigma+1} \int_{\R^N} (\DD \phi) |u|^{2 \sigma +2 } \, dx ,
\end{align*}
where $u_m=u_m(t,x)$ is defined in \eqref{def:um} above.
\end{lemma}

\begin{remarks*}
{\em 1) If we make formal substitution and take the unbounded function $\nabla \phi(x) = x$, we have $\pt_r^2 \phi \equiv 1$ and $\DD^2 \phi \equiv 0$. By applying the identity
$$
\int_0^\infty m^s \int_{\R^N} |\nabla u_m|^2 \, dx \, dm = s \| (-\DD)^{s/2} u \|_{L^2}^2
$$ 
for any $u \in \dot{H}^{s}(\R^N)$ (see \eqref{eq:Plancherel} below), we find the formal virial identity \eqref{eq:virial} by an elementary calculation.

2) From the proof given below and Lemma \ref{lem:u_m_inter}, we deduce the bound
$$
\begin{aligned}
\left | \int_0^\infty m^s  \int_{\R^N} \left \{  4   \, \overline{\pt_k u_m} (\pt^2_{kl} \phi) \pt_l u_m  -   (\DD^2 \phi)  |u_m|^2 \right \} \, dx \, d m \right |  \\ \lesssim  \| \nabla^2 \phi \|_{L^\infty} \| (-\DD)^{s/2} u \|_{L^2}^2  + \| \DD^2 \phi \|_{L^\infty}^{s} \| \DD \phi \|_{L^\infty}^{1-s} \| u \|_{L^2}^2 \lesssim C \| u \|_{H^s}^2,
\end{aligned}
$$
with some constant $C >0$ depending only on $\| \nabla \phi \|_{W^{3,\infty}}$.

3) The usage of the auxiliary function $u_m$ and Balakrshinan's representation formula \eqref{eq:Balakrish} for $(-\DD)^s$ is partly inspired by the joint work \cite{KrLeRa-13} of the third author. In \cite{KrLeRa-13}, the use of $u_m$ turns out to be helpful to show certain coercivity properties for the perturbative construction of minimal mass blowup solutions for the cubic half-wave equation in $N=1$ dimension.
}
\end{remarks*}

\begin{proof}[Proof of Lemma \ref{lem:virial_gen}] Define the (formally) self-adjoint differential operator
$$
\Gamma_{\phi} := -i ( \nabla \cdot \nabla \phi + \nabla \phi \cdot \nabla ),
$$
which acts on functions according to
$$
\Gamma_{\phi} f = -i \left ( \nabla \cdot ((\nabla \phi) f) + (\nabla \phi) \cdot ( \nabla f) \right ).
$$
We readily check that
$$
\MM_{\phi}[u(t)] = \langle u(t), \Gamma_{\phi} u(t) \rangle .
$$
By taking the time derivative and using the equation satisfied by $u(t)$, we get
\be \label{eq:Mder0}
\frac{d}{dt} \MM_{\phi}[u(t)] = \left \langle u(t), [ (-\DD)^s, i \Gamma_{\phi}] u(t) \right \rangle +  \left \langle u(t), [ -|u|^{2 \sigma}, i \Gamma_{\phi}] u(t) \right  \rangle,
\ee
where we recall that $[X,Y] \equiv XY - YX$ denotes the commutator of $X$ and $Y$.

By our regularity assumption on $u(t)$, we have $(-\DD)^s u(t) \in L^2(\R^N)$ and $\Gamma_{\phi} u(t) \in H^{2s-1}(\R^N) \subset L^2(\R^N)$ for $s \geq 1/2$. In particular, the terms above are well-defined a-priori. Next, we discuss the terms on the right side of \eqref{eq:Mder0} separately as follows. For notational ease, we simply write $u$ instead of $u(t)$ and $u(t,x)$ in what follows.

\medskip
{\bf Step 1 (Dispersive Term).} For $s \in (0,1)$, we have the formula
\be \label{eq:Balakrish}
(-\DD)^s = \frac{ \sin \pi s}{\pi} \int_0^\infty m^{s-1} \frac{-\DD}{-\DD+ m} \, dm,
\ee
which follows from spectral calculus applied to the self-adjoint operator $-\DD$ and the formula $x^s = \frac{\sin \pi s}{\pi} \int_0^\infty m^{s-1} \frac{x}{x+m} \, dm $ valid for any real number $x> 0$ and $s \in (0,1)$. In semigroup theory, the formula \eqref{eq:Balakrish} usually goes by the name {\em Balakrishnan's formula.} Next, we note the formal identity 
\be \label{eq:commId}
\left [\frac{A}{A+m},B \right ]  = \left [ \mathds{1} - \frac{m}{A+m}, B \right ] = -m \left [ \frac{1}{A+m}, B \right ] = m \frac{1}{A+m} [A,B] \frac{1}{A+m} ,
\ee
for operators $A \geq 0$ and $B$,  where $m > 0$ is any  positive real number.  By combining \eqref{eq:Balakrish} and \eqref{eq:commId} with $A=-\DD$, we obtain the formal commutator identity
\be \label{eq:Balabala1}
[(-\DD)^s, B ] = \frac{\sin (\pi s)}{\pi} \int_0^\infty m^{s} \frac{1}{-\DD + m} [-\DD, B] \frac{1}{-\DD+m} \, dm 
\ee
for any operator $B$. Next, we apply this identity to $B= i \Gamma_{\phi}$ and we use that
\be \label{eq:Balabala2}
[-\DD, i \Gamma_{\phi}] = - 4 \pt_k ( \pt^2_{kl} \phi ) \pt_l - \DD^2 \phi,
\ee
which follows from a direct calculation using the Leibniz rule. 

Let us now apply the formal identities above to the situation at hand. Indeed, let us first assume that $u \in C^\infty_c(\R^N)$ holds. We claim that
\be \label{eq:Bala1}
\langle u, [(-\DD)^s, i \Gamma_\phi] u \rangle = \int_0^\infty m^s \int_{\R^N}  \left \{ 4 \, \overline{\pt_k u_m} ( \pt^2_{kl}    \phi ) \pt_l u_m - (\DD^2 \phi) |u_m|^2 \right \} \, dx \, dm,
\ee
where $u_m = c_s (-\DD+m)^{-1} u$ with $m>0$ and the constant $c_s >0$ is defined in \eqref{def:cs}.  Now, for $u \in C^\infty_c(\R^N)$, we can readily apply formula \eqref{eq:Balakrish} (where the $m$-integral is a convergent Bochner integral) to express $(-\DD)^s u$. Furthermore, it is legitimate to use \eqref{eq:Balabala1} with \eqref{eq:Balabala2} and, by Fubini's theorem, we arrive at \eqref{eq:Bala1} provided that $u \in C^\infty_c(\R^N)$.

As a next step, we extend the identity \eqref{eq:Bala1} to any $u \in H^{2s}(\R^N)$ by the the following approximation argument. Let $u_n \in C^\infty_c(\R^N)$ be a sequence such that $u_n \to u$ strongly in $H^{2s}(\R^N)$. We easily see that $\langle u_n, [(-\DD)^s, i \Gamma_\phi] u_n \rangle \to \langle u, [(-\DD)^s, i \Gamma_\phi] uÊ\rangle$, which yields the left-hand side of \eqref{eq:Bala1}. Next, we claim that
\be \label{eq:conv11}
\lim_{n \to \infty} G[u_n, u_n] = G[u,u],
\ee
where we define the bilinear form
$$
G[f,g] := \int_0^\infty m^s \int_{\R^N} \overline{\pt_k f_m} (\pt^2_{kl} \phi) \pt_l g_m \, dx \, dm
$$
with $f_m = c_s(-\DD+m)^{-1}$ and $g_m = c_s (-\DD+m)^{-1} g$. Since $u_n \to u$ strongly in $H^{2s}(\R^N)$, the convergence \eqref{eq:conv11} clearly follows if we can show that
\be \label{eq:Gfg}
\left | G[f,g] \right | \lesssim \| \pt_{kl}^2 \phi \|_{L^\infty} \| (-\DD)^{s/2} f \|_{L^2} \| (-\DD)^{s/2} g \|_{L^2} .
\ee
To prove \eqref{eq:Gfg}, we first note that,  by using Plancherel's and Fubini's theorem, 
\be \label{eq:Plancherel}
\begin{aligned}
& \int_0^\infty m^s \int_{\R^N} |\nabla f_m|^2 \,dx \, dm = \int_{\R^N} \left ( \frac{\sin  \pi s}{\pi} \int_0^\infty \frac{m^s \, dm}{(|\xi|^2+m)^2}  \right ) |\xi|^2 |\widehat{f}(\xi)|^2 \, d \xi \\
& \phantom{\int_0^\infty m^s \int_{\R^N} |\nabla f_m|^2 \,dx \, dm} = \int_{\R^N} \left ( s |\xi|^{2s-2}\right ) |\xi|^2 |\widehat{f}(\xi)|^2 \, d \xi = s \| (-\DD)^{s/2} f \|_{L^2}^2
\end{aligned}
\ee
for arbitrary $f \in \dot{H}^s(\R^N)$. Next, we introduce the bilinear form
$$
H[f,g] := G[f,g] + \mu s \int_{\R^N} \overline{f} (-\DD)^s g \,dx \ \ \mbox{with} \ \ \mu := \text{ess-sup}_{x \in \R^N} \left \| ( \pt_{kl}^2 \phi)(x) \right  \|,
$$ 
where $\| A \|$ denotes the operator norm of a matrix $A \in \R^{N \times N}$. Thus from \eqref{eq:Plancherel} and by using the pointwise lower bound $\overline{ \pt_k f_m} (\pt_{kl}^2 \phi) \pt_l f_m \geq - \mu |\nabla f_m|^2$ we obtain that
$$
H[f,f] \geq - \mu \int_0^\infty m^s \int_{\R^N} |\nabla f_m|^2 \, dx \, dm + \mu s \| (-\DD)^{s/2} f \|_{L^2}^2 = 0.
$$
On the other hand, we have $\mu \lesssim \| \pt_{kl}^2 \phi \|_{L^\infty}$ and thus
$$
H[f,f] \lesssim \| \pt_{kl}^2 \phi \|_{L^\infty} \| (-\DD)^{s/2} f \|_{L^2}^2 .
$$
Since $H[f,g]$ is positive semidefinite, we have the Cauchy-Schwarz inequality $\left | H[f,g] \right | \leq \sqrt{H[f,f]} \sqrt{H[g,g]}$. Consequently, we deduce
\begin{align*}
\left | G[f,g] \right | & \leq \sqrt{ H[f,f]} \sqrt{H[g,g]} + \mu s \| (-\DD)^{s/2} f \|_{L^2} \| (-\DD)^{s/2} g \|_{L^2} \\ &
\lesssim \| \pt_{kl}^2 \phi \|_{L^\infty} \| (-\DD)^{s/2} f \|_{L^2} \| (-\DD)^{s/2} g \|_{L^2},
\end{align*}
which is the desired bound \eqref{eq:Gfg}.

To complete the proof of \eqref{eq:Bala1} for $u \in H^{2s}(\R^N)$, we need to show that
\be \label{eq:conv22}
\lim_{n \to \infty} K[u_n, u_n] = K[u,u]
\ee
for the bilinear form
$$
K[f,g] := \int_{0}^\infty m^s \int_{\R^N} (\DD^2 \phi) \overline{f_m} g_m \, dx \, dm.
$$
Indeed, by following the arguments in the proof of Lemma \ref{lem:u_m_inter}, we obtain
$$
\left | K[f,g] \right | \lesssim \| \DD^2 \phi \|_{L^\infty}^s \| \DD \phi \|_{L^\infty}^{1-s} \| f \|_{L^2} \| g \|_{L^2},
$$
from which we readily deduce that \eqref{eq:conv22} holds.

[In fact, the previous arguments allow us to extend identity \eqref{eq:Bala1} to any $u \in H^s(\R^N)$. However, as previously remarked, the extension of the identity in Lemma \ref{lem:virial_gen} to $H^s$-valued solutions $u(t)$ would require an approximation argument by $H^{2s}$-valued solutions $u(t)$, which we do not study here.]

\medskip
{\bf Step 2 (Nonlinear Term).} This part of the proof is analogous to the classical NLS. In fact, an integration by parts yields 
\begin{align*}
\left \langle u, [-|u|^{2 \sigma}, i \Gamma_\phi] u \right \rangle & = - \left \langle u, [|u|^{2 \sigma},  \nabla \phi \cdot \nabla + \nabla \cdot \nabla \phi] u \right \rangle \\
& = 2 \int_{\R^N} |u|^2 \nabla \phi \cdot  \nabla (|u|^{2 \sigma}) =  - \frac{2 \sigma}{\sigma+1} \int_{\R^N} (\DD \phi) |u|^{2 \sigma +2}, 
\end{align*}
where we also made use of the identity $\nabla (|u|^{2 \sigma+2}) = \frac{\sigma+1}{\sigma} \nabla (|u|^{2 \sigma}) |u|^2$.

This completes the proof of Lemma \ref{lem:virial_gen}.
\end{proof}

\subsection{Localized Virial Estimate for Radial Solutions} \label{subsec:MR}
We now apply the previous formula for $\MM_{\phi}[u(t)]$ when $\phi(x)$ is a suitable approximation of the unbounded function $a(x)=\frac{1}{2} |x|^2$ and hence $\nabla a(x) = x$. This choice will yield a localized virial identity that will be used to prove blowup for radial solutions of fractional NLS.

Let $\phi : \R^N \to \R$ be  as above. In addition, we assume that $\phi=\phi(r)$ is radial and satisfies 
\be  \label{def:phi}
\phi(r) = \begin{dcases*} r^2/2 & \quad for $r \leq 1$ \\ \mbox{const}. & \quad for $r \geq 10$ \end{dcases*} \quad \mbox{and} \quad  \mbox{$\phi''(r) \leq 1$ for $r \geq 0$}.
\ee
For $R > 0$ given, we define the rescaled function $\phi_R : \R^N \to \R$ by setting
\be
\phi_R(r) := R^2 \phi \left ( \frac{r}{R} \right ) .
\ee
We readily verify the inequalities
\be \label{ineq:phi1}
1 - \phi_R''(r) \geq 0, \quad 1- \frac{\phi_R'(r)}{r} \geq 0, \quad N - \DD \phi_R(r) \geq 0 \quad \mbox{for all $r \geq 0$}.
\ee
Indeed, this first inequality follows from $\phi_R''(r) = \phi''(r/R) \leq 1$. We obtain the second inequality by integrating the first inequality on $[0,r]$ and using that $\phi_R'(0) = 0$. Finally, we find that $N - \DD \phi_R(r) = 1-\phi_R''(r) + (N-1)\{1 - \frac{1}{r} \phi_R'(r)\} \geq 0$ holds thanks to the first two inequalities in \eqref{ineq:phi1}.  

For later use, we record the following properties of $\phi_R$, which can be easily checked:
\be \label{eq:phi2}
\left \{
\begin{aligned}
& \nabla \phi_R(r) =   R \phi' \left ( \frac{r}{R} \right ) \frac{x}{|x|} = \begin{dcases*} x & for $r \leq R$ \\ 0 & for $r \geq 10R$ \end{dcases*} ; \\ 
& \| \nabla^j \phi_R \|_{L^\infty} \lesssim R^{2-j} \quad \mbox{for $0 \leq j \leq 4$} \, ; \\
& \mathrm{supp} \, ( \nabla^j \phi_R ) \subset \begin{dcases*} \left \{ |x| \leq 10 R \right \} & for $j=1,2$ \\ \left \{ R \leq |x| \leq 10 R \right \} & for $3 \leq j \leq 4$ \end{dcases*} .
\end{aligned}
\right .
\ee


For the time evolution of the localized virial $\MM_{\phi_R}[u(t)]$ with $\phi_R$ as above, we have the following estimate.

\begin{lemma}[Localized Radial Virial Estimate] \label{lem:MR_radial}
Let $N \geq 2$, $s \in (1/2,1)$, and assume in addition that $u(t,x)$ is a radial solution of \eqref{eq:fNLS}. We then have
\begin{align*}
\frac{d}{dt} \MM_{\phi_R}[u(t)] & \leq 4 \sigma N E[u_0] -  2 (  \sigma N - 2s) \| (-\DD)^{s/2} u(t) \|_{L^2}^2 \\
& \quad + C  \cdot \left ( R^{-2s} + C R^{ - \sigma ( N-1) + \eps s} \| (-\DD)^{s/2} u(t) \|_{L^2}^{(\sigma/s) + \eps} \right ),
\end{align*}
for any $0 < \eps < (2s-1)\sigma/s$. Here $C=C(\| u_0 \|_{L^2}, N, \eps, s, \sigma) >0$ is some constant that only depends on $\|u_0 \|_{L^2}, N, \eps, s$ and $\sigma$.
\end{lemma}

\begin{remark*}
{\em Note that we assume the strict inequality $s > 1/2$ here. In the limiting case $s=1/2$, the radial Sobolev inequality \eqref{ineq:Cho} below fails to hold, which is however needed in the proof to control the error induced by the nonlinearity. 
}
\end{remark*}

\begin{proof}
As usual, we shall often omit the time variable $t$ in the argument of $u(t,x)$ in the following due to notational convenience.
First, we recall the Hessian of a radial function $f : \R^N \to \C$ can be written as
$$
\pt^2_{kl} f = \left ( \delta_{kl} - \frac{x_l x_k}{r^2} \right ) \frac{\pt_r f}{r} + \frac{x_k x_l}{r^2} \pt_r^2 f.
$$
Thus, we can rewrite the first term on the right-hand side in Lemma \ref{lem:virial_gen} as follows.
$$
4  \int_0^\infty m^s \int_{\R^N}   \overline{\pt_k u_m} (\pt_{kl}^2 \phi_R) \pt_l u_m \, dx \, dm  = 4 \int_0^\infty m^s \int_{\R^N} ( \pt_r^2 \phi_R) |\nabla u_m|^2 \, dx \, dm.
$$
Recalling \eqref{eq:Plancherel} and inequality \eqref{ineq:phi1}, we deduce that
$$
\begin{aligned}
& 4  \int_0^\infty m^s \int_{\R^N} \overline{\pt_k u_m}   (\pt_{kl}^2 \phi_R) \pt_l u_m \, dx \, dm & \\
& =  4s \| (-\DD)^{s/2} u(t) \|_{L^2}^2 - 4 \int_0^\infty m^s \int_{\R^N} \left ( 1 - \pt_r^2 \phi_R \right ) |\nabla u_m|^2 \, dx \, dm \\
& \leq 4s \| (-\DD)^{s/2} u(t) \|_{L^2}^2 .
\end{aligned}
$$
Moreover, from Lemma \ref{lem:u_m_inter} we have the bound
$$
\left | \int_0^\infty m^s \int_{\R^N} ( \DD^2 \phi_R ) |u_m|^2 \, dx \, dm \right | \lesssim \| \DD^2 \phi_R \|_{L^\infty}^s \| \DD \phi_R \|_{L^\infty}^{1-s} \| u \|_{L^2}^2 \lesssim R^{-2s} ,
$$
where we also used the properties of $\phi_R$ and the conservation of $L^2$-mass of $u(t)$. 

For the last term on right-hand side in Lemma \ref{lem:virial_gen}, we recall that $\DD \phi_R(r) - N \equiv 0$ on $ \{ r \leq R \}$ and we thus obtain that
\begin{align*}
- \frac{2\sigma}{\sigma+1} \int_{\R^N} (\DD \phi_R) |u|^{2 \sigma+2} \, dx  & = - \frac{  2\sigma N}{\sigma+1} \int_{\R^N} |u|^{2 \sigma+2} \, dx \\
& \quad  - \frac{2 \sigma}{\sigma+1} \int_{|x| \geq R} ( \DD \phi_R - N) |u|^{2 \sigma+2} \, dx.
\end{align*}
Next, we recall from \cite{ChOz-09} the fractional radial Sobolev (generalized Strauss) inequality
\be \label{ineq:Cho}
\sup_{x \neq 0} |x|^{\frac{N}{2}-\alpha} |u(x)| \leq C(N,\alpha) \| (-\DD)^{\alpha/2} u \|_{L^2} \quad \
\ee
for all radial functions $u \in \dot{H}^{\alpha}(\R^N)$ provided that $1/2 < \alpha < N/2$. Now, let $0 < \eps < (2s-1)\sigma/ s$ and set $\alpha = \frac{1}{2} + \eps \frac{s}{2 \sigma}$, which implies that $1/2 < \alpha < s < N/2$. From the interpolation inequality $\| (-\DD)^{\alpha/2} u \|_{L^2} \leq \| u \|_{L^2}^{1-\alpha/s} \| (-\DD)^{s/2} u \|_{L^2}^{\alpha/s} \lesssim \| (-\DD)^{s/2} u \|_{L^2}^{\alpha/s}$ and estimate \eqref{ineq:Cho}, we deduce
\begin{align*}
\int_{|x| \geq R} |u|^{2 \sigma+2} \, dx & \leq \| u \|_{L^2}^2 \| u \|_{L^\infty(|x| \geq R)}^{2 \sigma} \lesssim C(N,\alpha, \eps) R^{-2\sigma(\frac{N}{2}-\alpha) } \| (-\DD)^{\alpha/2} u \|_{L^2}^{2 \sigma} \\
& \lesssim C(N, \alpha, \eps) R^{-2 \sigma( \frac{N}{2} - \alpha)} \| (-\DD)^{s/2} u \|_{L^2}^{2 \sigma \alpha/s} \\
& = C(N, \alpha,\eps) R^{-\sigma (N-1)+ \eps s } \| (-\DD)^{s/2} u \|_{L^2}^{(\sigma/s) + \eps}.
\end{align*}

In summary, we have shown that
\begin{align*}
\frac{d}{dt} \MM_{\phi_R}[u(t)] & \leq 4s \| (-\DD)^{s/2} u(t) \|_{L^2}^2 - \frac{2 \sigma N}{\sigma+1} \int_{\R^N} |u(t,x)|^{2 \sigma +2} \, dx \\
& \quad + C \cdot \left ( R^{-2s} + C R^{-\sigma(N-1) + \eps s} \| (-\DD)^{s/2} u(t) \|_{L^2}^{(\sigma/s)+\eps} \right ) \\
& =  4 \sigma N E[u_0] -  2 (  \sigma N - 2s) \| (-\DD)^{s/2} u(t) \|_{L^2}^2 \\
& \quad + C \cdot \left ( R^{-2s} + C R^{ - \sigma ( N-1) + \eps s} \| (-\DD)^{s/2} u(t) \|_{L^2}^{(\sigma/s) + \eps} \right ),
\end{align*}
for any $0 < \eps < (2s-1)\sigma/s$ with some constant $C=C(\| u_0 \|_{L^2}, N,\eps, s,  \sigma) >0$. Note that  we used the conservation of energy $E[u(t)]$  in the last step. The proof of Lemma \ref{lem:MR_radial} is now complete.
\end{proof}

For the proof of Theorem \ref{thm:L2super} (ii) below (which deals with the $L^2$-critical case), we shall need the following refined version of Lemma \ref{lem:MR_radial} involving the nonnegative radial functions
\be \label{def:psi12}
\psi_{1,R}(r) := 1 - \pt_r^2 \phi_R(r) \geq 0 \quad \mbox{and} \quad \psi_{2,R}(r) := N - \DD \phi_R(r) \geq 0.
\ee

\begin{lemma}[A Refined Version of Lemma \ref{lem:MR_radial}] \label{lem:MR_radial_fine} Under the hypotheses of Lemma \ref{lem:MR_radial} and $\sigma=2s/N$, we  have that
\begin{align*}
\frac{d}{dt} \MM_{\phi_R}[u(t)] & \leq 8 s E[u_0] - 4 \int_0^{\infty} m^s \int_{\R^N} \left \{ \psi_{1,R} - c(\eta) \psi_{2,R}^{\frac{N}{2s}} \right \} |\nabla u_m|^2 \,dx \, dm \\
& \quad + \cO \left ( (1 + \eta^{-\beta}) R^{-2s} + \eta( 1 + R^{-2} + R^{-4}) \right ),
\end{align*}
for every $\eta > 0$ and $R >0$,  where $c(\eta) = \eta/(N + 2 s)$ and $\beta = 2s/(N - 2s)$.
\end{lemma}

\begin{proof}
For notational convenience, we write $\psi_1 = \psi_{1,R}$ and $\psi_2 = \psi_{2,R}$ in the following. Inspecting the proof of Lemma \ref{lem:MR_radial}, we immediately get
\be \label{eq:thomas_master}
\begin{aligned}
& \frac{d}{dt} \MM_{\phi_R}[u(t)]  = 8 s E[u_0] - 4 \int_0^\infty m^s \int_{\R^N}  \psi_1 |\nabla u_m|^2 \, dx \, dm \\
& \phantom{ \frac{d}{dt} \MM_{\phi_R}[u(t)] =} + \frac{4 s}{N + 2 s} \int_{\R^N} \psi_2 |u|^{\frac{4s}{N}+2}  \, dx+ \cO(R^{-2s}). 
\end{aligned}
\ee
We divide the rest of the proof into following steps.

\medskip
{\bf Step 1 (Control of Nonlinearity).} Recall that $\mathrm{supp} \, \psi_2 \subset \{ |x| \geq R \}$. We apply the radial Sobolev inequality \eqref{ineq:Cho} to the radial function $\psi_2^{\frac{N}{4s}} u \in H^s(\R^N)$ and use that $\| u \|_{L^2} \lesssim 1$, which together yields
\be \label{ineq:thomas_strauss}
\begin{aligned}
& \int_{\R^N} \psi_2 |u|^{\frac{4s}{N}+2} \, dx = \int_{|x| \geq R} ( \psi_2^{\frac{N}{4s}} |u| )^{\frac{4s}{N}} |u|^2 \,dx \leq \| \psi_2^{\frac{N}{4s}} u \|_{L^\infty(|x| \geq R)}^{\frac{4s}{N}} \| u \|_{L^2}^2 \\
& \phantom{ \int_{\R^N} \psi_2 |u|^{\frac{4s}{N}+2} \, dx} \lesssim R^{-\frac{2s}{N}(N-2s)} \| (-\DD)^{s/2} (\psi_2^{\frac{N}{4s}} u ) \|_{L^2}^{\frac{4s}{N}} \\
& \phantom{\int_{\R^N} \psi_2 |u|^{\frac{4s}{N}+2} \, dx} \leq \eta \| (-\DD)^{s/2} ( \psi_2^{\frac{N}{4s}} u ) \|_{L^2}^2 + \cO (\eta^{- \beta} R^{-2 s}), \quad \beta = \frac{2s}{N - 2s} ,
\end{aligned}
\ee
where in the last step we used Young's inequality $a b \lesssim \eta a^q + \eta^{-p/q} b^p$ with $1/p+1/q=1$ such that $q=N/2s$, $\beta = p/q$, and $\eta > 0$ is an arbitrary number. For notational convenience,  let us define $\chi := \psi_2^{\frac{N}{4s}}$. From the identity \eqref{eq:Plancherel} we recall that
\be \label{eq:thomas}
s \| (-\DD)^{s/2} ( \chi u) \|_{L^2}^2 = \int_0^\infty m^s \int_{\R^N} |\nabla ( \chi u )_m |^2 \, dx \, dm ,
\ee
where we denote
$$
(\chi u)_m =c_s\frac{1}{-\DD + m} ( \chi u )
$$
for $m >0$ and $c_s$ as in \eqref{def:cs} above. To estimate the right-hand side of \eqref{eq:thomas}, we split the $m$-integral in the regions $\{ 0 < m \leq 1 \}$ (low frequencies) and $\{ m \geq 1\}$ (high frequencies).  

To estimate the contribution in the low-frequency region, we notice that
\be \label{ineq:mlow}
 \left | \int_0^1 m^s \int_{\R^N} \left | \frac{\nabla}{-\DD+m} (\chi u) \right  |^2 \, dx \, dm \right |   \leq \int_0^1 m^{s-1} \| \chi u \|_{L^2}^2 \, dm  \lesssim 1,
\ee
where we make use of the bounds $\| \frac{\nabla}{-\DD + m} \|_{L^2 \to L^2} \leq m^{-1/2}$ and $\| \chi \|_{L^\infty} \lesssim1 $. To control the right-hand side of \eqref{eq:thomas} in the high frequency region $\{ m \geq 1\}$, we need a more elaborate argument  worked out in the next step.

\medskip
{\bf Step 2 (Control of High Frequencies $m \geq 1$).} By using the commutator identity
$\left [ \frac{1}{-\DD+m}, \chi \right ] = \frac{1}{-\DD +m} [\DD, \chi] \frac{1}{-\DD+m}$,
we conclude
\begin{align*}
\nabla ( \chi u )_m & = \nabla (\chi u_m) + c_s \nabla \left [ \frac{1}{-\DD+m}, \chi \right ] u = \chi \nabla u_m + \nabla \chi u_m + \frac{\nabla}{-\DD+m}[\DD, \chi] u_m.
\end{align*}
with $c_s = \sqrt{\sin (\pi s)/\pi}$ defined  in \eqref{def:cs}. Thus we get
\begin{align*}
& \left | \int_1^\infty m^s \int_{\R^N} \left | \frac{\nabla}{-\DD+m} [\DD, \chi] u_m \right |^2 \,dx \, dm \right |  \lesssim \int_1^\infty m^{s-1} \left ( \| \nabla \chi \cdot \nabla u_m \|_{L^2}^2 + \| \Delta \chi u_m \|_{L^2}^2 \right ) \, dm \\
& \lesssim \int_1^\infty m^{s-1} \left \{ \| \nabla \chi \|_{L^\infty}^2 \left \| \frac{\nabla}{-\DD+m} u \right \|_{L^2}^2 + \| \DD \chi \|_{L^\infty}^2 \left \| \frac{1}{-\DD+m} u \right \|_{L^2}^2\right \} \, dm \\
& \lesssim  \int_1^\infty \left ( m^{s-2} \| \nabla \chi \|_{L^\infty}^2  + m^{s-3} \| \DD \chi \|_{L^\infty}^2   \right ) \, dm \lesssim \frac{\| \nabla \chi \|_{L^\infty}^2}{1-s} + \frac{\| \DD \chi \|_{L^\infty}^2}{2-s},
\end{align*}
where we used that $[\DD, \chi] = 4 (\nabla \chi) \cdot \nabla + \DD \chi$ as well as the estimates $\| \frac{\nabla}{-\DD +m} \|_{L^2 \to L^2} \leq m^{-1/2}$ and $\| \frac{1}{-\DD +m} \|_{L^2 \to L^2} \leq m^{-1}$ and conservation of mass in the last line. Similarly, we get
$$
\left | \int_1^\infty m^s \int_{\R^N} | \nabla \chi u_m |^2 \, dx \, dm \right | \lesssim \frac{\| \nabla \chi \|_{L^\infty}^2}{1-s}.
$$
Recalling that $\chi = \psi_2^{\frac{N}{4s}}$ with $\psi_2 = N - \DD \phi_R$, the properties \eqref{eq:phi2} are seen to imply that $\| \nabla \chi \|_{L^\infty} \lesssim R^{-1}$ and $\| \DD \chi \|_{L^\infty} \lesssim R^{-2}$. Thus we can summarize the estimates found above and \eqref{ineq:mlow} to conclude that
\be \label{eq:thomas2}
s \| (-\DD)^{s/2} (\chi u ) \|_{L^2}^2 = \int_0^\infty m^s \int_{\R^N} \chi^2 |\nabla u_m|^2 \, dx \, dm + \cO ( 1+R^{-2} + R^{-4} ).
\ee

\medskip
{\bf Step 3 (Conclusion).} If we now combine \eqref{eq:thomas2} with \eqref{ineq:thomas_strauss}, we obtain
\begin{align*}
\int_{\R^N} \psi_2 |u|^{\frac{4s}{N}+2} \, dx  & =  \frac{\eta}{s} \int_0^\infty m^s \int_{\R^N}  \chi^2 |\nabla u_m|^2 \, dx \, dm  \\
& \quad + \cO \left ( \eta^{-\beta} R^{-2s} + \eta (1 + R^{-2} + R^{-4}) \right ).
\end{align*}
By inserting this back into \eqref{eq:thomas_master} and setting $c(\eta)=\eta/(N+2s)$, we complete the proof of Lemma \ref{lem:MR_radial_fine}.
\end{proof}

\section{Radial Blowup in $\R^N$: Proof of Theorem \ref{thm:L2super}}

In this section, we prove Theorem \ref{thm:L2super}. We discuss the cases (i) and (ii) as follows.

\subsection{Proof of Theorem \ref{thm:L2super}, Case (i)}
Let $N \geq 2$ and $s \in (1/2,1)$. We consider the $L^2$-supercritical case when $0 < s_c \leq s$ and we impose the extra (technical) condition that $\sigma < 2s$ holds (see below for details on this condition). Furthermore, we suppose that
$$
u \in C\big ( [0,T); H^{2s}(\R^N) \big)
$$
is a radial solution of \eqref{eq:fNLS}. Let $\phi_R(r)$ with $R> 0$ be a radial cutoff function on $\R^N$ as introduced in Subsection \ref{subsec:MR} above. For notational convenience, we shall write
$$
\MM_R[u(t)] := \MM_{\phi_R}[u(t)]
$$
for the localized virial of $u(t)$. We organize the rest of the proof as follows.

\medskip
{\bf Case 1: $E[u_0] < 0$.}  Let us define $\delta :=   \sigma N -2s > 0$. From Lemma \ref{lem:MR_radial} with $\eps > 0$ sufficiently small and fixed, we deduce the inequality (with $o_R(1) \to 0$ as $R \to +\infty$ uniformly in $t$):
\be \label{ineq:MR_vir1}
\begin{aligned}
& \frac{d}{dt} \MM_{R}[u(t)]  \leq 4  \sigma N E[u_0] - 2 \delta \| (-\DD)^{s/2} u(t) \|_{L^2}^2 + o_R(1) \cdot \left (1 +  \|(-\DD)^{s/2} u(t) \|_{L^2}^{(\sigma/s)+\eps} \right )  \\
& \phantom{\frac{d}{dt} \MM_{R}[u(t)]}  \leq 2 \sigma N E[u_0] - \delta \| (-\DD)^{s/2} u(t) \|_{L^2}^2 \quad \mbox{for all $t \in [0,T)$},
\end{aligned}
\ee
provided that $R \gg 1$ is taken sufficiently large. In the last step, we used that $E[u_0] < 0$, Young's inequality, and that $\sigma/s + \eps < 2$ when $\eps > 0$  is sufficiently small. [At this point, the condition $\sigma < 2s$ is needed.]

With estimate \eqref{ineq:MR_vir1} at hand, we can now adapt the strategy of Ogawa-Tstutsumi \cite{OgTs-91} to the setting of fractional NLS with focusing $L^2$-supercritical nonlinearity. Suppose $u(t)$ exists for all times $t \geq 0$, i.\,e., we can take $T=+\infty$. From \eqref{ineq:MR_vir1} and $E[u_0] < 0$ it follows that $\frac{d}{dt} \MM_R[u(t)] \leq -c$ with some constant $c> 0$. By integrating this bound, we conclude that $\MM_R[u(t)] < 0$ for all $t \geq t_1$ with some time sufficiently large time $t_1 \gg 1$. Thus, if we integrate \eqref{ineq:MR_vir1} on $[t_1, t]$, we obtain
\be \label{ineq:MR_good2}
\MM_R[u(t)] \leq -\delta \int_{t_1}^t \| (-\DD)^{s/2} u(\tau) \|_{L^2}^2 \, d\tau \leq 0 \quad \mbox{for all $t \geq t_1$}.
\ee
On the other hand, we use Lemma \ref{lem:MR_estimate} and $L^2$-mass conservation to find that
\be \label{ineq:MR_good}
\begin{aligned}
& |\MM_R[u(t)]|  \lesssim C(\phi_R)  \left ( \| |\nabla|^{1/2} u(t) \|_{L^2}^2 +  \| |\nabla|^{1/2} u(t) \|_{L^2} \right ) \\
& \phantom{ |\MM_R[u(t)]|} \lesssim C(\phi_R) \left (  \| (-\DD)^{s/2} u(t) \|_{L^2}^{1/s} + \| (-\DD)^{s/2} u(t) \|_{L^2}^{1/2s} \right ),
\end{aligned}
\ee
where we also used the interpolation estimate $\| |\nabla|^{1/2} u \|_{L^2} \leq \| (-\DD)^{s/2} u \|_{L^2}^{1/2s} \| u \|_{L^2}^{1-1/2s}$ for $s > 1/2$. Next, we claim the lower bound
\be \label{ineq:Hslower}
\| (-\DD)^{s/2} u(t) \|_{L^2} \gtrsim 1 \quad \mbox{for all $t \geq 0$}.
\ee
Indeed, suppose this bound was not true. Thus we have that $\| (-\DD)^{s/2} u(t_k) \|_{L^2} \to 0$ for some sequence of times $t_k \in [0,\infty)$. However, by $L^2$-mass conservation and the Gagliardo-Nirenberg inequality, this implies that $\| u(t_k) \|_{L^{2 \sigma+2}} \to 0$ as well. Hence we get $E[u(t_k)] \to 0$, which is a contradiction to  $E[u(t)] = E[u_0] < 0$. Thus we deduce that \eqref{ineq:Hslower} holds.

If we now combine the lower bound \eqref{ineq:Hslower} with \eqref{ineq:MR_good}, we  find 
\be
|\MM_R[u(t)]|  \lesssim C(\phi_R) \| (-\DD)^{s/2} u(t) \|_{L^2}^{1/s} .
\ee
Thus we conclude from \eqref{ineq:MR_good2} that
\be
\MM_{R}[u(t)] \lesssim -C(\phi_R) \int_{t_1}^t |\MM_{R}[u(\tau)]|^{2s} \, d \tau \quad \mbox{for all $t \geq t_1$}.
\ee
By using this nonlinear integral inequality, a straightforward argument yields the bound $\MM_{R}[u(t)] \lesssim -C(\phi_R) |t - t_*|^{1-2s}$ for $s > 1/2$ with some finite $t_* < + \infty$. Therefore we have $\MM_{R}[u(t)] \to -\infty$ as $t \uparrow t_*$. Hence the solution $u(t)$ cannot exist for all times $t \geq 0$ and consequently we must have that $T < +\infty$ holds. 

{\bf Case 2: $E[u_0] \geq 0$.} Suppose that $E[u_0] \geq 0$ and that we have
\be \label{ineq:payne}
\left \{ \begin{array}{ccc} 
E[u_0]^{s_c} M[u_0]^{s-s_c} & <  &E[Q]^{s_c} M[Q]^{s-s_c}  , \\[1ex]
\| (-\DD)^{s/2} u_0 \|_{L^2}^{s_c} \| u_0 \|_{L^2}^{s-s_c} & > & \| (-\DD)^{s/2} Q \|_{L^2}^{s_c} \| Q \|_{L^2}^{s-s_c} .
\end{array}
\right .
\ee
Recall our convention that for the energy-critical case $s_c=s$, we set $M[Q]^{s-s_c} = (M[Q])^0 = 1$ although, the ground state $Q$ may fail to be in $L^2(\R^N)$ for $s =s_c$; see Section \ref{sec:cutoff} below.

From the conservation of energy and $L^2$-mass combined with Gagliardo-Nirenberg inequality \eqref{ineq:GN} (when $s_c < s$) or Sobolev's inequality \eqref{ineq:Sob} (when $s_c=s$) we get
\be \label{ineq:E0payne}
E[u_0] = \frac{1}{2} \| (-\DD)^{s/2} u(t) \|_{L^2}^2 - \frac{1}{2 \sigma +2} \| u(t) \|_{L^{2 \sigma+2}}^{2 \sigma+2} \geq F ( \| (-\DD)^{s/2} u(t) \|_{L^2} ),
\ee 
where the function $F : [0, \infty) \to \R$ is defined as
\be
F(y) := \frac{y^2}{2} - \frac{C_{N, \sigma, s}}{2 \sigma +2} ( M[u_0] )^{\frac{\sigma}{s}(s-s_c)} y^{2 + 2\sigma \frac{s_c}{s}} , \ \ \text{with} \ \  {2 + 2\sigma \frac{s_c}{s} = \frac{\sigma N}{s}},
\ee
where $C_{N,\sigma,s} > 0$ denotes the optimal constant for the Gagliardo-Nirenberg inequality \eqref{ineq:GN} if $s_c < s$ or Sobolev's inequality \eqref{ineq:Sob} if $s_c=s$. We readily verify that $F(y)$ has a unique global maximum
\be
F(y_{\max}) = \frac{s_c}{N} y^2_{\mathrm{max}},
\ee
which is attained at
\be
y_{\max} = (K_{N, \sigma,s})^{\frac{1}{s_c}} M[u_0]^{- \frac{s - s_c}{2 s_c}} \quad \mbox{with} \quad K_{N, \sigma,s} = \left ( \frac{2s(\sigma+1)}{ \sigma  NC_{N,\sigma,s}} \right)^{\frac{s}{2\sigma}}  .
\ee
Next, by Proposition \ref{prop:Q}, we have
$$
K_{N,\sigma,s} = \| (-\DD)^{s/2} Q \|_{L^2}^{s_c} \| Q \|_{L^2}^{s-s_c} = \left ( \frac{s_c}{N} \right )^{-\frac{s_c}{2}} E[Q]^{\frac{s_c}{2}} M[Q]^{\frac{s-s_c}{2}} .
$$
Thus condition \eqref{ineq:payne} tells us that
$$
E[u_0] < F(y_{\mathrm{max}}) \quad \mbox{and} \quad \| (-\DD)^{s/2} u_0 \|_{L^2} > y_{\mathrm{max}} .
$$
By continuity in time, we deduce that
\be \label{ineq:paynelow}
\| (-\DD)^{s/2} u(t) \|_{L^2} > y_{\mathrm{max}} \quad \mbox{for all $t \in [0,T)$}.
\ee
Indeed, suppose this bound was not true. Then, by continuity, there is some time $t_* \in (0,T)$ such that $\| (-\DD)^{s/2} u(t_*) \|_{L^2} = y_{\max}$. But this contradicts \eqref{ineq:E0payne}, since $E[u_0] < F(y_{\max})$. Therefore the lower bound \eqref{ineq:paynelow} holds.

Next, we pick $\eta > 0$ sufficiently small to ensure that
$$
E[u_0]^{s_c} M[u_0]^{s-s_c} \leq (1- \eta)^{s_c} E[Q]^{s_c} M[Q]^{s-s_c} .
$$
From estimate \eqref{ineq:paynelow} we obtain by an elementary calculation that
$$
2 \delta (1-\eta) \| (-\DD)^{s/2} u(t) \|_{L^2}^2 \geq 4 \sigma N E[u_0] \quad \mbox{for all $t \in [0,T)$},
$$
where we recall that $\delta = \sigma N - 2 s > 0$. By inserting this bound into the differential inequality from Lemma \ref{lem:MR_radial}, we get
\be
\begin{aligned}
\frac{d}{dt} \MM_{R}[u(t)] & \leq 4  \sigma N E[u_0] - 2 \delta \| (-\DD)^{s/2} u(t) \|_{L^2}^2 \\
& \quad + o_R(1) \cdot \left ( 1 + \| (-\DD)^{s/2} u(t) \|_{L^2}^{\sigma/s + \eps} \right ) \\ 
& \leq - \left (\delta \eta + o_R(1) \right ) \| (-\DD)^{s/2} u(t) \|_{L^2}^2 + o_R(1),
\end{aligned}
\ee
with $o_R(1) \to 0$ as $R \to \infty$ uniformly in $t$, where we have chosen $\eps > 0$ small enough such that $\sigma/s + \eps < 2$ (which is possible, since $\sigma < 2s$ by assumption). Choosing $R \gg 1$ sufficiently large and using \eqref{ineq:paynelow} again, we thus conclude
\be \label{ineq:payne3}
\frac{d}{dt} \MM_R[u(t)] \leq - \frac{\delta \eta}{2} \|(-\DD)^{s/2} u(t) \|_{L^2}^2 \quad \mbox{for all $t \in [0,T)$}.
\ee
Suppose now that $T= +\infty$ holds. Since $\| (-\DD)^{s/2} u(t) \|_{L^2} > y_{\max} > 0$ for all $t \geq 0$, we see from \eqref{ineq:payne3} that $\MM_R[u(t)] < 0$ for all $t \geq t_1$ with some sufficiently large time $t_1 \gg 1$. Hence, by integrating on $[t_1, t]$, we obtain
$$
\MM_{R}[u(t)] \leq - \frac{\delta \eta}{2} \int_{t_1}^t \| (-\DD)^{s/2} u(\tau) \|_{L^2}^2 \, d \tau \leq 0 \quad \mbox{for all $t \geq t_1$}.
$$
By following exactly the steps after \eqref{ineq:MR_good2} above, we deduce that $u(t)$ cannot exist for all times $t \geq 0$.

The proof of Theorem \ref{thm:L2super}, Case (i) is now complete.

\subsection{Proof of Theorem \ref{thm:L2super},  Case (ii)}

Let $N \geq 2$, $s \in (1/2,1)$, and we consider the $L^2$-critical exponent $\sigma=2s/N$. We assume that 
$$
u \in C \big ([0,T); H^{2s}(\R^N) \big)
$$
is a radial solution of \eqref{eq:fNLS} with negative energy
$$
E[u_0] < 0.
$$

Let $\phi_R(r)$ be a radial cutoff function as introduced in Subsection \ref{subsec:MR} above. Recall the definitions of the functions $\psi_{1,R}(r)$ and $\psi_{2,R}(r)$ in \eqref{def:psi12}, depending on the function $\phi_R(r)$. Furthermore, as in Lemma \ref{lem:MR_radial_fine}, we set $c(\eta) = \eta/(N+2s)$ for $\eta > 0$. As shown in Section \ref{sec:cutoff} below, we can choose $\phi_R(r)$ and $\eta > 0$ sufficiently small such that 
$$
\psi_{1,R}(r) - c(\eta) ( \psi_{2,R}(r))^{\frac{N}{2s}} \geq 0 \quad \mbox{for all $r > 0$,}
$$   
and for all $R >0$. 

Thus if we choose $\eta \ll 1$ sufficiently small and then $R \gg 1$ sufficiently large, we can apply Lemma \ref{lem:MR_radial_fine} to deduce that
\be \label{ineq:L2crit}
\frac{d}{dt} \MM_R[u(t)] \leq 4 s E[u_0] \quad \mbox{for $t \in [0,T)$},
\ee
where we write $\MM_{\phi_R}[u(t)] = \MM_{R}[u(t)]$ for notational convenience. Next, we suppose that $u(t)$ exists for all times $t \geq 0$, i.\,e., we can take $T =+ \infty$. From \eqref{ineq:L2crit} we infer that 
\be \label{ineq:L2crit_2}
\MM_R[u(t)] \leq -c t \quad \mbox{for $t \geq t_0$},
\ee
with some sufficiently large time $t_0 > 0$ and some constant $c > 0$ depending only on $s$ and $E[u_0] < 0$. On the other hand, if we invoke Lemma \ref{lem:MR_estimate}, we see that
\be \label{ineq:L2crit_3}
\begin{aligned}
\left | \MM_R[u(t)] \right | & \lesssim C(\phi_R) \left ( \| |\nabla|^{1/2} u(t) \|_{L^2}^2 + \| u(t) \|_{L^2} \| |\nabla|^{1/2} u(t) \|_{L^2} \right ) \\
& \lesssim C(\phi_R) \left ( \| |\nabla|^{1/2} u(t) \|_{L^2}^2 + 1 \right )  \lesssim C(\phi_R) \left ( \| (-\DD)^{s/2} u(t) \|_{L^2}^{1/s} + 1 \right ),
\end{aligned}
\ee
where we also used the conservation of $L^2$-mass of $u(t)$ together with the interpolation  estimate $\| |\nabla|^{1/2} u \|_{L^2} \leq \| (-\DD)^{s/2} u \|_{L^2}^{1/2s} \| u \|_{L^2}^{1-1/2s}$ for $s > 1/2$. By combining \eqref{ineq:L2crit_3} and \eqref{ineq:L2crit_2}, we finally get
\be
\| (-\DD)^{s/2} u(t) \|_{L^2} \geq C t^s \quad \mbox{for $t \geq t_*$,}
\ee
with some sufficiently large time $t_* > 0$ and some constant $C > 0$ depending only on $u_0, s$, and $N$. 

The proof of Theorem \ref{thm:L2super} is now complete. \hfill $\qed$

\section{Blowup on Bounded Domains: Proof of Theorem \ref{thm:dom}}

Let $N \geq 1$, $s \in (1/2,1)$, and $0 < s_c \leq s$. Suppose that $\Om \subset \R^N$ is a bounded and star-shaped domain with smooth boundary $\pt \Om$. Without loss of generality we can assume that $\Om$ is star-shaped with respect to the origin $0 \in \Om$, i.\,e., we have $\alpha x \in \Om$ for any $x \in \Om$ and any $\alpha \in [0,1]$. In the following, we assume that
$$
u \in C([0,T); D((-\DD)^s))
$$ 
solves problem \eqref{eq:fnlsO}. 

\subsection{Virial Law on $\Om$}
We define the virial of $u(t)$ as
$$
\MM_\Om[u(t)] := 2 \, \mathrm{Im} \int_{\Om} \overline{u}(t)( x \cdot \nabla u (t))  \, dx .
$$
To see that $\MM_{\Om}[u(t)]$ is well-defined, we recall that $D((-\DD)^s) = H^{s(2s)}_2(\overline{\Om})$ and the inclusion $H^{s(2s)}_2(\overline{\Om}) \subset H^{s(1)}_2(\overline{\Om})$ (since $s > 1/2$) from \cite[Example 7.2~and Eqn.~(1.31)]{Gr-15}. Moreover, we have the equality\footnote{For $0 \leq s \leq 1$, the spaces $H^s_0(\overline{\Om})$ introduced in \cite{Gr-15}  coincide with the space $H^s_0(\Om)$ defined in \eqref{def:Hs0}.} $H^{s(1)}(\overline{\Om})=\dot{H}_0^1(\Om)$ by \cite[Theorem 5.4]{Gr-15} using that $s-1 \in (-1/2,1/2)$. Thus $u(t) \in D((-\DD)^s)$ implies that $u(t) \in \dot{H}^1_0(\Om)$, whence it follows $u(t) \in H^1_0(\Om)$ by Poincar\'e's inequality, since $\Om$ is bounded.  
 
We now establish the following key inequality for the time evolution of the virial on $\Om$.
\begin{lemma} \label{lem:MOm}
For any $t \in [0,T)$, we have
$$
\frac{d}{dt} \MM_\Om[u(t)] \leq 4 \sigma N E_\Om[u_0] -  2 (  \sigma N - 2s) \int_{\Om} \overline{u}(t) (-\DD)^{s} u(t) \, dx .
$$
\end{lemma}

\begin{proof}
For notational convenience, we denote $\MM_\Om(t) := \MM_\Om[u(t)]$ in the following. Furthermore, we write
$\langle f, g \rangle = \int_{\Om} \overline{f} g \, dx$ for the inner product in $L^2(\Om)$.

{\bf Step 1.} 
First, we show that $t \mapsto \MM_\Om(t)$ is of class $C^1$ and calculate its derivative. Indeed, let $h \neq 0$ and assume $t, t+h \in [0,T)$. We find that
\begin{align*}
& \MM_\Om(t+h) - \MM_\Om(t) \\
&= 2 \, \mathrm{Im}  \left \langle u(t+h)-u(t), x \cdot \nabla u(t+h) \right \rangle  + 2 \, \mathrm{Im} \left \langle u(t), x \cdot \nabla (u(t+h) -u(t)) \right \rangle  \\
& = 2 \, \mathrm{Im} \left \langle u(t+h)-u(t), x \cdot \nabla u(t+h) \right \rangle - 2 \, \mathrm{Im} \left \langle x \cdot \nabla u(t), u(t+h)-u(t) \right \rangle \\
& \quad - 2 N \, \mathrm{Im} \left  \langle u(t), u(t+h) -u(t) \right  \rangle,
\end{align*}  
where we used the identity $x \cdot \nabla f = \nabla \cdot (xf) - N f$ and we integrated by parts recalling that $u(t), u(t+h) \in  D((-\DD)^s) \subset H^1_0(\Om)$ for $s > 1/2$. Since $u \in C^0([0,T); H^1_0(\Om)) \cap C^1([0,T); L^2(\Om))$, we can take the limit $\frac{1}{h} [\MM_{\Om}(t+h)- \MM_{\Om}(t)]$ as $h \to 0$ to deduce that $\MM_{\Om}(t)$ belongs to $C^1$ with its derivative given by
\be \label{eq:Mder}
\frac{d}{dt} \MM_{\Om}(t) = 4 \, \mathrm{Im} \left \langle \pt_t u(t), x \cdot \nabla u(t) \right \rangle +2 N \, \mathrm{Im} \left \langle \pt_t u(t), u(t) \right \rangle =: (I) + (II),
\ee
using that $\mathrm{Im} \, \langle f,g \rangle = - \mathrm{Im} \, \langle g, f \rangle$. 

\medskip
{\bf Step 2.}
We analyze the term $(I)$ as follows. Using that $\pt_t u= -i (-\DD)^s u +i |u|^{2 \sigma} u$, we get
\begin{align*}
(I) & = 4 \, \mathrm{Re} \int_{\Om}   (-\DD)^s{\overline u}(t)  \, ( x \cdot \nabla u(t)) \, dx  - 4 \, \mathrm{Re} \int_{\Om} |u(t)|^{2 \sigma} \overline{u}(t) \, (x \cdot \nabla u(t)) \, dx   \\
& =  4 \, \mathrm{Re} \int_{\Om}   (-\DD)^s{\overline u}(t)  \, ( x \cdot \nabla u(t)) \, dx - 2 \int_{\Om} x \cdot ( |u(t)|^{2 \sigma} \nabla (|u(t)|^2)) \, dx .
\end{align*}
Here  we also used the simple fact that $\overline{(-\DD)^s u}= (-\DD)^s \overline{u}$. Next, we apply the Pohozaev-type estimate in Lemma \ref{lem:poho} and use that $|u|^{2 \sigma} \nabla |u|^2 = \frac{1}{\sigma+1} \nabla (|u|^{2 \sigma+2})$ and integrate by parts to find that
\be 
(I) \leq (4s - 2N) \int_{\Om} u(t) (-\DD)^s \overline{u}(t) \, dx + \frac{2 N}{\sigma+1} \int_{\Om} |u(t)|^{2 \sigma+2} \,dx .
\ee

Next, for the second term on the right-hand side in \eqref{eq:Mder}, a direct calculation shows 
$$
(II)  =  2 N \left (  \int_{\Om} u (-\DD)^s \overline{u} \, dx - \int_{\Om} |u(t)|^{2 \sigma+2} \, dx \right ).
$$
Going back to \eqref{eq:Mder}, we conclude that
\begin{align*}
\frac{d}{dt} \MM_{\Om}[u(t)]  = (I) + (II) &  \leq 4s \int_{\Om} u(t) (-\DD)^s \overline{u}(t) \, dx - \frac{2 N \sigma}{\sigma+1} \int_{\Om} |u(t)|^{2 \sigma+2} \, dx \\
& = 4 \sigma N  E_{\Om}[u(t)] - 2( \sigma N - 2s) \int_\Om \overline{u}(t) (-\DD)^s u(t) \, dx,
\end{align*}
where the last step follows from conservation of energy. This completes the proof of Lemma \ref{lem:MOm}.
\end{proof}

\subsection{Proof of Theorem \ref{thm:dom}}

With Lemma \ref{lem:MOm} at hand, we can now follow  the arguments used in the proof of Theorem \ref{thm:L2super} above. For the reader's convenience, we provide the details adapted to the case of a bounded domain. 

Let $\delta = \sigma N - 2s > 0$. Suppose that $E_\Om[u_0] < 0$ and assume that $T = +\infty$ holds, i.\,e., the solution $u(t)$ exists for all times $t \geq 0$. By integrating the inequality in Lemma \ref{lem:MOm}, we deduce that $\MM_\Om[u(t)] \leq 0$ for all $t \geq t_1$ with some sufficiently large time $t_1>0$ and that
\be \label{ineq:MOm1}
\MM_{\Om}[u(t)] \leq  - \delta \int_{t_1}^t \int_{\Om} u(s) (-\DD)^s \overline{u}(s) \, dx \, ds \leq 0 \quad \mbox{for all $t \geq t_1$} .
\ee
Now, let $R > 0$ be a sufficiently large radius such that $\overline{\Om} \subset B_R(0)$. Take a function $\phi \in C^\infty_c(\R^N)$ with $\mathrm{supp} \, \phi \subset B_{2R}(0)$ and $\nabla \phi(x) \equiv x$ on $B_R(0)$. Since $u(t) \in H^s_0(\Omega)$ (i.\,e., $u(t) \in H^s(\R^N)$ with $u \equiv 0$ on $\R^N \setminus \Om$), we find
$$
\MM_{\Om}[u(t)] = 2\, \mathrm{Im} \int_{\Om} \overline{u}(t) (x \cdot \nabla u(t)) \, dx = 2 \, \mathrm{Im} \int_{\R^N} \overline{u}(t) (\nabla \phi \cdot \nabla u(t)) \, dx.
$$
Applying Lemma \ref{lem:MR_estimate} and using that $\|u(t) \|_{L^2(\R^N)} = \| u(t) \|_{L^2(\Om)} \lesssim 1$ by $L^2(\Om)$-mass conservation, we get
\begin{align*}
\left | \MM_{\Om}[u(t)] \right | & \lesssim C(\phi) \left (\| |\nabla|^{1/2} u(t) \|_{L^2(\R^N)}^2 +  \| |\nabla|^{1/2} u(t) \|_{L^2(\R^N)} \right ) \\
& \lesssim C(\phi) \left ( \| |\nabla|^{1/2} u(t) \|_{L^2(\R^N)}^2 + 1 \right )  \lesssim C(\phi) \left ( \| (-\DD)^{s/2} u(t) \|_{L^2(\R^N)}^{1/s} + 1 \right ) \\
& \lesssim C(\phi ) \left ( \left ( \int_{\Om} \overline{u}(t) (-\DD)^s u(t) \, dx \right )^{\frac{1}{2s}} + 1 \right ),
\end{align*}
using the interpolation estimate $\| |\nabla|^{1/2} u \|_{L^2} \leq \| (-\DD)^{s/2} u \|_{L^2}^{1/2s} \| u \|_{L^2}^{1-1/2s}$ for $s \geq 1/2$, as well as $\| (-\DD)^{s/2} u \|_{L^2}^2 = \int_{\Om} \overline{u} (-\DD)^s u \, dx$ for $u \in H^s_0(\Om)$. Next, by adapting the arguments using energy considerations given  in the proof of Theorem \ref{thm:L2super}, we get the uniform lower bound
\be
\int_{\Om} \overline{u}(t) (-\DD)^s u(t) \, dx \gtrsim 1 \quad \mbox{for all $t \geq 0$}.
\ee
Hence, we conclude that
\be
\left |\MM_{\Om}[u(t)] \right | \lesssim C(\phi) \left ( \int_{\Om} \overline{u}(t) (-\DD)^s u(t) \, dx \right )^{\frac{1}{2s}} ,
\ee
for any $t \geq 0$. Thus by going back to \eqref{ineq:MOm1}, we obtain
\be
\MM_{\Om}[u(t)] \lesssim - C (\phi) \int_{t_1}^t |\MM_{\Om}[u(\tau)]|^{2s} \, d \tau \quad \mbox{for all $t \geq t_1$}.
\ee
Since $2s > 1$, this integral inequality implies that  $\MM_{\Om}[u(t)] \lesssim -C(\phi) |t - t_*|^{1-2s}$ tends to $-\infty$  as $t \nearrow t_*$ with some finite $t_* < + \infty$. Therefore, the solution $u(t)$ fails to exist for all times $t \geq 0$ and hence $T < +\infty$ must hold.

The proof of Theorem \ref{thm:dom} is now complete. \hfill $\qed$

\begin{remark*}
{\em For the half-wave case $s=1/2$ and $0 < s \leq s_c$, the arguments in the proof of Theorem \ref{thm:dom} formally yield the following result: If $u \in C([0,T); D((-\DD)^{1/2}))$ solves \eqref{eq:fnlsO} with negative energy $E_{\Om}[u_0] < 0$, then $u(t)$ either blows up in finite time or $u(t)$ blows up in infinite time such that
$$
\| | \nabla|^{1/2} u(t) \|_{L^2(\Om)} \gtrsim e^{a t} \ \ \mbox{for $t \geq 0$},
$$
with some constant $a> 0$. However, we have $D((-\DD)^{1/2})) = H^{1/2(1)}_2(\overline{\Om})$ and it is only known that $H^{1/2(1)}_2(\overline{\Om}) \subset H_0^{1-\eps}(\Om)$ for any $\eps \in (0,1]$; see \cite[Theorem 5.4]{Gr-15}. Therefore, it is not guaranteed that the pairing $\langle (-\DD)^{1/2} u, x \cdot \nabla u \rangle$ appearing above is well-defined for $u \in D((-\DD)^{1/2})$.}
\end{remark*}

\begin{appendix}

\section{Various Estimates}

\begin{lemma} \label{lem:MR_estimate}
Let $N \geq 1$ and suppose $\phi : \R^N \to \R$ is such that $\nabla \phi \in W^{1,\infty}(\R^N)$. Then, for all $u \in H^{1/2}(\R^N)$, it holds that
$$
\left | \int_{\R^N} \overline{u}(x) \nabla \phi(x) \cdot \nabla u(x) \, dx \right | \leq  C \left ( \| |\nabla|^{1/2} u \|_{L^2}^2 +   \| u \|_{L^2} \| |\nabla|^{1/2} u \|_{L^2} \right ),
$$
with some constant $C > 0$ that depends only on $\| \nabla \phi \|_{W^{1,\infty}}$ and $N$.
\end{lemma}

\begin{proof}
By writing $\nabla = |\nabla|^{1/2} \frac{\nabla}{|\nabla|} |\nabla|^{1/2}$ and using the Cauchy-Schwarz inequality, we estimate
\be \label{ineq:A1}
\begin{aligned}
\left | \int_{\R^N} \overline{u}(x) \nabla \phi(x) \cdot \nabla u(x) \, dx \right | & = \left | \left \langle |\nabla|^{1/2} ((\nabla \phi) u), \frac{\nabla}{|\nabla|} |\nabla|^{1/2} u \right \rangle \right | \\ &  \leq \| |\nabla|^{1/2} ((\nabla \phi) u) \|_{L^2} \left \| \frac{\nabla}{|\nabla|} |\nabla|^{1/2} u \right \|_{L^2} \\ &  \lesssim \| |\nabla|^{1/2}(  (\nabla \phi) u) \|_{L^2} \| |\nabla|^{1/2} u \|_{L^2},
\end{aligned}
\ee
where in the last step we used the fact that the Riesz projector $\nabla /|\nabla|$ is a bounded operator on $L^2(\R^N)$. Now we claim that
\be \label{ineq:LiLo}
\| |\nabla|^{1/2} (( \nabla \phi )u ) \|_{L^2} \lesssim \| \nabla \phi \|_{W^{1, \infty}}  \left ( \| |\nabla|^{1/2} u \|_{L^2} + \| u \|_{L^2} \right ).
\ee
Indeed, this estimate can be deduced from adapting the proof of \cite[Theorem 7.16]{LiLo-01} as follows. We note
\begin{align*}
& \| |\nabla|^{1/2} (( \nabla \phi ) u ) \|_{L^2}^2 =  \mbox{(const)} \cdot \iint_{\R^N \times \R^N} \frac{|\nabla \phi(x) u(x) - \nabla \phi(y) u(y)|^2}{|x-y|^{N+1}} \,dx \, dy \\
& \lesssim \iint_{\R^N \times \R^N} \frac{ |u(x)-u(y)|^2 |\nabla \phi(y)|^2}{|x-y|^{N+1}} \, dx \, dy +  \iint_{\R^N \times \R^N} \frac{ |\nabla \phi(x) - \nabla \phi(y)|^2 |u(x)|^2}{|x-y|^{N+1}} \, dx \,dy  \\
& \lesssim \| \nabla \phi \|_{L^\infty}^2 \iint_{\R^N \times \R^N} \frac{|u(x)-u(y)|^2}{|x-y|^{N+1}} \, dx \, dy + \| \nabla^2 \phi \|_{L^\infty}^2 \iint_{|x-y| \leq 1} \frac{1}{|x-y|^{N-1}} |u(x)|^2 \, dx \, dy \\
& \quad + \| \nabla \phi \|_{L^\infty}^2 \iint_{|x-y| > 1} \frac{1}{|x-y|^{N+1}} | u(x) |^2 \,dx \,dy \\
& \lesssim \| \nabla \phi \|_{W^{1,\infty}}^2 \left ( \| |\nabla|^{1/2} u \|_{L^2}^2 + \| u \|_{L^2}^2 \right ),
 \end{align*}
 whence \eqref{ineq:LiLo} follows by taking the square root. If we insert \eqref{ineq:LiLo} back into \eqref{ineq:A1}, we finish the proof.
\end{proof}

\begin{lemma} \label{lem:u_m_inter}
Let $N \geq 1$, $s \in (0,1)$, and suppose $\phi : \R^N \to \R$ with $\Delta \phi \in W^{2, \infty}(\R^N)$. Then, for all $u \in L^2(\R^N)$, we have
$$
\left | \int_0^\infty m^s \int_{\R^N} (\Delta^2 \phi) |u_m|^2 \, dx \, dm \right | \lesssim  \| \DD^2 \phi \|_{L^\infty}^s  \| \DD \phi \|_{L^\infty}^{1-s}\| u \|_{L^2}^2.
$$
\end{lemma}

\begin{remark*}{\em 
A direct application of H\"older's inequality together with \eqref{eq:Plancherel} yields the bound 
$$
\left | \int_0^\infty m^s \int_{\R^N} (\Delta^2 \phi) |u_m|^2 \, dx \, dm \right | \lesssim \| \DD^s \phi \|_{L^\infty} \| |\nabla|^{s-1} u \|_{L^2}^2.
$$
However, such a bound in terms of the negative order Sobolev norm $\| u \|_{\dot{H}^{s-1}}$ would be of no use to us. 
}
\end{remark*}

\begin{proof}
We extend the proof in \cite[Lemma B.3]{KrLeRa-13} to $N \geq 1$ and $s \in (0,1)$. Thus we split the $m$-integral into $\int_0^\Lambda \ldots + \int_{\Lambda}^\infty \ldots$ with a parameter  $\Lambda > 0$ to be determined below.  First, we integrate by parts in $x$ twice and use H\"older's inequality to find that
\begin{align*}
& \left | \int_0^\Lambda m^s \int_{\R^N} (\DD^2 \phi) |u_m|^2 \, dx \, dm \right | \\
&  =  \left | \int_0^\Lambda m^s \int_{\R^N} (\DD \phi) \left \{ (\DD  u_m) \overline{u_m} + u_m (\DD \overline{u_m}) + 2 \nabla u_m \cdot \nabla \overline{u_m} \right \}  dx \, dm \right | \\
& \lesssim \| \DD \phi \|_{L^\infty} \int_0^\Lambda m^s \left (  \| \Delta u_m \|_{L^2} \| u_m \|_{L^2} + \| \nabla u_m \|_{L^2}^2 \right )  dm \\
& \lesssim \| \DD \phi \|_{L^\infty} \| u \|_{L^2}^2\left (  \int_0^\Lambda m^{s-1} \, dm \right ) \lesssim \| \DD \phi \|_{L^\infty} \| u \|_{L^2}^2 \Lambda^s .
\end{align*}
Here, we have also used the bounds
$$
\| \Delta u_m \|_{L^2}  \lesssim \| u \|_{L^2}, \quad \| \nabla u_m \|_{L^2} \lesssim m^{-1/2} \| u \|_{L^2}, \quad \| u_m \|_{L^2} \lesssim m^{-1} \| u \|_{L^2},
$$
which are immediate consequences of the definition $u_m = c_s \cdot (-\DD+m)^{-1} u$ (as in \eqref{def:um}, \eqref{def:cs} above) and Plancherel's identity. Furthermore, we find that
\begin{align*}
& \left | \int_{\Lambda}^\infty m^s \int_{\R^N} (\DD^2 \phi) |u_m|^2 \, dx \, dm \right | \lesssim \| \DD^2 \phi \|_{L^\infty} \left ( \int_\Lambda^\infty m^s \| u_m \|_{L^2}^2 \, dm \right ) \\
& \lesssim \| \DD^2 \phi \|_{L^\infty} \| u \|_{L^2}^2 \left ( \int_\Lambda^\infty m^{s-2} \, dm \right )  \lesssim \| \DD^2 \phi \|_{L^\infty} \| u \|_{L^2}^2 \Lambda^{s-1} .
\end{align*}
In summary, we have shown that
$$
\left | \int_0^\infty m^s \int_{\R^N} (\DD^2 \phi) |u_m|^2 \, dx \, dm \right | \lesssim \left ( \| \DD \phi \|_{L^\infty} \Lambda^s + \| \DD^2 \phi \|_{L^\infty} \Lambda^{s-1} \right ) \| u \|_{L^2}^2
$$
for arbitrary $\Lambda > 0$. By minimizing the right-hand side with respect to $\Lambda$, we are led to the choice
$\Lambda = \frac{1-s}{s} \frac{\| \DD^2 \phi \|_{L^\infty}}{\| \DD \phi \|_{L^\infty}}$, which yields the desired bound.
\end{proof}

The next result provides a Pohozaev-type estimate for $(-\DD)^s$ (with exterior Dirichlet conditions) on bounded and star-shaped domains $\Om$. 

\begin{lemma}[Pohozaev-Type Estimate] \label{lem:poho}
Let $N \geq 1$ and $s \in (0,1)$. Suppose that $\Om \subset \R^N$ is a bounded domain that is star-shaped with respect to the origin $0 \in \Om$. Then, for all $u \in H^1_0(\Om)$ with $(-\DD)^s u \in L^2(\Om)$, we have the inequality
$$
\mathrm{Re} \, \int_{\Om} ( x \cdot \nabla u) (-\DD)^s \overline{u}  \, dx \leq \left ( \frac{2s-N}{2} \right ) \int_{\Om} u (-\DD)^s \overline{u} \, dx .
$$
\end{lemma}

\begin{remark*}
{\em The idea of the proof goes back to Ros-Oton and Serra \cite{RoSe-14}, where in fact an identity is shown for $u$ that satisfy additional regularity conditions. In that case, the boundary term (given by  the one-sided derivative $\frac{d}{d\lambda} |_{\lambda \to 1^+} I_{\lambda}$ below) can be worked out explicitly. In our setting, we do not need this explicit form and we can allow for less strict regularity assumptions on $u$.
}
\end{remark*}

\begin{proof}
We adapt the arguments in \cite{RoSe-14}; see also \cite{RoSe-15}.
For $\lambda > 1$, we set $u_\lambda (x) = u(\lambda x)$. Since in particular $u \in H^1(\Om)$, we can show that
\be
\frac{u_\lambda - u}{\lambda -1} \ \ \weakto \ \ x \cdot \nabla u \ \ \mbox{weakly in $L^2(\Om)$ as $\lambda \downarrow 1$},
\ee
see, e.\,g., the proof of \cite[Lemma 4.2]{RoSe-15}. Since we also have $(-\DD)^s \overline{u} \in L^2(\Om)$, we deduce
\be
\int_{\Om} (x \cdot \nabla u) (-\DD)^s \overline{u} \, dx = \left .  \frac{d}{d \lambda} \right |_{\lambda \downarrow 1}  \int_{\Om} u_\lambda (-\DD)^s \overline{u} \, dx
\ee
Using that $u_\lambda, u \in H^s(\R^N)$ with $u_\lambda \equiv 0$ on $\R^N \setminus \Om$, we find
\begin{align*}
\int_{\Om} u_\lambda (-\DD)^s \overline{u} \, dx & = \int_{\R^N} u_\lambda (-\DD)^s \overline{u} \, dx  = \int_{\R^N} (-\DD)^{s/2} u_\lambda (-\DD)^{s/2} \overline{u} \, dx \\
& = \lambda^{\frac{2s-N}{2}} \int_{\R^N} w_{\sqrt{\lambda}} \overline{w}_{1/\sqrt{\lambda}} \, dy ,
\end{align*}
where $y = \sqrt{\lambda} x$, $w(x) = (-\DD)^{s/2}u(x)$, and $w_\lambda(x) = w(\lambda x)$. If we take real parts, we thus obtain
\begin{align*}
& \mathrm{Re} \int_{\Om} (x \cdot \nabla u) (-\DD)^s \overline{u} \, dx =  \left . \frac{d}{d \lambda} \right |_{\lambda \downarrow 1} \mathrm{Re} \left \{ \lambda^{\frac{2s -N}{2}} \int_{\R^N} w_{\lambda} \overline{w}_{1/\sqrt{\lambda}} \, dy \right \} \\
& = \left ( \frac{2s -N}{2} \right ) \mathrm{Re} \, \int_{\R^N} w \overline{w} \, dx + \left . \frac{d}{d \lambda} \right |_{\lambda \downarrow 1} I_{\sqrt{\lambda}} \\ & =  \left ( \frac{2s -N}{2} \right ) \mathrm{Re} \, \int_{\Om} u (-\DD)^s \overline{u} \, dx + \frac{1}{2} \left . \frac{d}{d \lambda} \right |_{\lambda \downarrow 1} I_{\lambda},
\end{align*}
where 
$$
I_\lambda = \mathrm{Re} \int_{\R^N} w_{\lambda} \overline{w}_{1/\lambda} \, dy.
$$
Now  the Cauchy-Schwarz inequality yields $I_\lambda \leq \| w_{\lambda} \|_{L^2} \| w_{1/\lambda} \|_{L^2} = \| w \|_{L^2}^2 = I_1$. Therefore,
$$
\left . \frac{d}{d \lambda} \right |_{\lambda \downarrow 1}  I_{\lambda} \leq 0.
$$
This completes the proof of Lemma \ref{lem:poho}.
\end{proof}

\section{Ground States and Cutoff Functions}

\label{sec:cutoff}

\subsection{Pohozaev Identities for Ground States}

Let $N \geq 1$, $s \in (0,1)$ and $\sigma > 0$. Recall the definition of the scaling index $s_c = \frac{N}{2} - \frac{s}{\sigma}$. In the energy-subcritical case $s_c < s$, we have the following Gagliardo-Nirenberg inequality
\be \label{ineq:GN}
\| u \|_{L^{2 \sigma+2}}^{2 \sigma+2} \leq C_{N, \sigma, s} \| (-\DD)^{s/2} u \|_{L^2}^{\frac{\sigma N}{s}} \| u \|_{L^2}^{2 \sigma + 2 - \frac{\sigma N}{s}},
\ee
valid for all $u \in H^s(\R^N)$. Here $C_{N, \sigma, s} > 0$ denotes the best constant. From \cite{FrLe-13,FrLeSi-15}, we recall existence and uniqueness (modulo symmetries) of optimizers $Q \in H^s(\R^N)$ for \eqref{ineq:GN}, which we refer to as {\em ground states}. Moreover as shown in \cite{FrLe-13,FrLeSi-15}, we can choose $Q=Q(|x|)> 0$ to be radially symmetric, strictly positive, and strictly decreasing in $|x|$. The function $Q$ is smooth and it can be rescaled to solve the equation 
\be \label{eq:Q}
(-\DD)^s Q + Q - Q^{2 \sigma+1} = 0 \quad \mbox{in $\R^N$}.
\ee
We have the following identities for the ground state $Q$.

\begin{prop} \label{prop:Q}
It holds that
$$
K_{N,\sigma,s} = \| (-\DD)^{s/2} Q \|_{L^2}^{s_c} \| Q \|_{L^2}^{s-s_c} = \left ( \frac{s_c}{N} \right)^{-\frac{s_c}{2}} E[Q]^{\frac{s_c}{2}} M[Q]^{\frac{s-s_c}{2}},
$$
where
$$
K_{N, \sigma,s} = \left ( \frac{2 s(\sigma+1)}{\sigma N C_{N,\sigma,s}} \right )^{\frac{s}{2 \sigma}} .
$$
\end{prop}

\begin{proof}
By integrating equation \eqref{eq:Q} against $Q$ and $x \cdot \nabla Q$ (where standard arguments show that $x \cdot \nabla Q \in H^s(\R^N)$, see \cite{FrLeSi-15}), we obtain the Pohozaev identities
$$
\| (-\DD)^{s/2} Q \|_{L^2}^2 + \| Q \|_{L^2}^2 - \| Q \|_{L^{2 \sigma+2}}^{2 \sigma+2} = 0,
$$
$$
\left ( \frac{2s-N}{2} \right ) \| (-\DD)^{s/2} Q \|_{L^2}^2 - \frac{N}{2} \| Q \|_{L^2}^2 + \frac{N}{2 \sigma+2} \| Q \|_{L^{2 \sigma+2}}^{2 \sigma+2}  = 0.
$$
Here we used that $ \langle x \cdot \nabla Q, (-\DD)^s Q \rangle = \left ( \frac{2s-N}{2} \right ) \| (-\DD)^s Q \|_{L^2}^2$ and $\langle x \cdot \nabla Q, Q \rangle = -\frac{N}{2} \| Q \|_{L^2}^2$. By using the two Pohozaev identities above together with the fact that $Q$ turns \eqref{ineq:GN} into an equality, the rest of the proof follows from straightforward calculations.
\end{proof}

Finally, we consider the energy-critical case $s_c=s$, i.\,e., we have $\sigma= \sigma_*:= \frac{2s}{N-2s}$, which requires that we are in space dimension $N > 2s$. In this case, we are lead to the Sobolev inequality
\be \label{ineq:Sob}
\| u \|_{L^{2 \sigma_* +2}}^{2 \sigma_* +2} \leq C_{N,\sigma} \| (-\DD)^{s/2} u \|_{L^2}^{ 2\sigma_* +2}
\ee
valid for all $u \in \dot{H}^s(\R^N)$, where $C_{N, \sigma} > 0$ denotes the best constant. Existence and uniqueness (modulo symmetries) of optimizers for \eqref{ineq:Sob} are classical facts; see, e.\,g, \cite{Li-83} via the equivalent problem of optimizing the weak Young inequality. In fact, the set of optimizers $Q \in \dot{H}^s(\R^N)$ for \eqref{ineq:Sob} are known in closed form and are given by
$$
Q_{\lambda, \mu, a}(x) = \lambda \cdot \left ( \frac{ 1 }{\mu^2 + |x-a|^2 } \right )^{\frac{N-2s}{2}}
$$
with parameters $\lambda \in \C \setminus \{ 0 \}$, $\mu> 0$, and $a \in \R^N$. Without loss of generality we can take $a=0$ and choose $\lambda$ real-valued and positive and pick $\mu>0$, so that $Q(x)= Q(|x|) > 0$ is radial and positive optimizer of \eqref{ineq:Sob} which solves
\be \label{eq:QSob}
(-\DD)^{s} Q - Q^{ \frac{N+2s}{N-2s}} =0 \quad \mbox{in $\R^N$}.
\ee
Note that $Q \in L^2(\R^N)$ if and only if $N > 4s$.

\begin{prop} \label{prop:QSob}
For the Sobolev optimizer $Q \in \dot{H}^s(\R^N)$ as above, we have
$$
K_{N,s} = \| (-\DD)^{s/2} Q \|_{L^2}^s = \left ( \frac{s}{N} \right )^{-\frac{s}{2}} E[Q]^{\frac{s}{2}} \ \  \mbox{with} \ \ K_{N,s} = \left ( \frac{1}{C_{N,s}} \right )^{\frac{N-2s}{4}} .
$$
\end{prop}

\begin{proof}
If we integrate \eqref{eq:QSob} against $Q$, we find $\| (-\DD)^{s/2} Q \|_{L^2}^2 = \| Q \|_{L^{2 \sigma_* +2}}^{2 \sigma_* +2}$ with $\sigma_* = 2s/(N-2s)$. Since $Q$ also optimizes \eqref{ineq:Sob}, we obtain the desired result by a straightforward calculation.  
\end{proof}

\subsection{Cutoff Function for $L^2$-Critical Case}
To construct a suitable virial function $\phi(r)$ for the $L^2$-critical case, we can adapt the choice made in \cite{OgTs-91} used for classical NLS. Let $g \in W^{3,\infty}(\R^N)$ be a radial function such that
\be
g(r) = \begin{dcases*} r & \quad for $0 \leq r \leq 1$, \\
r - (r-1)^3 & \quad for $1 < r \leq 1 + 1/\sqrt{3}$, \\
\mbox{$g(r)$ smooth and $g'(r) < 0$} & \quad for $1+1/\sqrt{3} < r < 10$, \\
0 & \quad for $r \geq 10$. \end{dcases*}
\ee
We define the radial function $\phi(r)$ by setting
\be
\phi(r) := \int_0^r g(s) \, ds,
\ee
It is elementary to check that $\phi(r)$ defined above satisfies assumption \eqref{def:phi}. Recall that we set $\phi_R(r) = R^2 \phi(r/R)$ for $R > 0$ given. Furthermore, recall the definitions of the nonnegative functions $\psi_{1,R}(r)=1- \pt_r^2 \phi_R(r)$ and $\psi_{2,R}(r)=N-\DD \phi_R(r)$ from \eqref{def:psi12}. Let $c(\eta) = \eta/(N+2s)$ for $\eta > 0$. We claim that if $\eta > 0$ sufficiently small and any $R >0$, we have
\be \label{ineq:psi12}
\psi_{1,R}(r) - c(\eta) ( \psi_{2,R}(r) )^{\frac{N}{2s}} \geq 0 \quad \mbox{for all $r \geq 0$}. 
\ee
To prove \eqref{ineq:psi12}, we argue as follows. First, by scaling, we can assume $R=1$ without loss of generality. Let us put $\psi_1(r) = \psi_{1,R=1}(r)$ and $\psi_{2}(r) = \psi_{2,R=1}(r)$. Note that $\psi_{1,R}(r) \equiv \psi_{2,R}(r) \equiv 0$ for $0 \leq r \leq R$ and hence \eqref{ineq:psi12} is trivially true in that region. Next, we observe that
$$
\psi_1(r) \geq 1 , \quad |\psi_2(r)|=|N - \DD \phi(r)| \leq C \quad \mbox{for $r \geq 1+1/\sqrt{3}$},
$$ 
with some constant $C > 0$.  Thus we can choose $\eta > 0$ sufficiently small such that \eqref{ineq:psi12} holds for $r \geq 1+1/\sqrt{3}$. Finally, a computation yields that 
$$
\psi_1(r) = 3 (r-1)^2, \quad |\psi_2(r)|^{\frac{N}{2s}} = |N - \DD \phi(r) |^{\frac{N}{2s}} \leq C (r-1)^{\frac{N}{s}} \quad \mbox{for $1 \leq r \leq 1+ 1/\sqrt{3}$},
$$
with some constant $C > 0$. Since $N/s \geq 2$, we deduce that \eqref{ineq:psi12} holds  in the region $1 \leq r \leq 1 + 1/\sqrt{3}$ too, provided that $\eta > 0$ is sufficiently small. 

\end{appendix}

\bibliographystyle{amsplain}
\bibliography{fNLSBlowup}

\end{document}